\documentclass[12pt]{article}
\UseRawInputEncoding 
\usepackage{amsfonts, amssymb, amsmath, amsthm, amssymb}
\usepackage{esint}
\usepackage{pifont}
\usepackage{bbding}
\usepackage{latexsym}
\usepackage{mathrsfs}
\usepackage{graphicx}
\usepackage{lineno}
\usepackage[colorlinks, linkcolor=blue, anchorcolor=blue, citecolor=blue]{hyperref}
\usepackage{verbatim}
\allowdisplaybreaks 
\usepackage[left=1in, right=1in, bottom=1.4in]{geometry}

\newtheorem{theorem}{\indent Theorem}[section]
\newtheorem{coro}{\indent Corollary}[section]

\newtheorem{lemma}{\indent Lemma}[section]

\numberwithin{equation}{section}
\newcommand{\mc}{\mathcal}

\newcommand{\R}{\mathbb{R}}
\newcommand{\wt}{\widetilde}

\newcommand{\ka}{\kappa}
\newcommand{\be}{\beta}

\newcommand{\va}{\varepsilon}
\newcommand{\al}{\alpha}
\newcommand{\pa}{\partial}

\newcommand{\lm}{\lambda}

\newcommand{\de}{\delta}
\newcommand{\De}{\Delta}
\newcommand{\na}{\nabla}

\date{} 
\begin{document}

\title{Homogenization of locally periodic parabolic operators with non-self-similar scales}

\author{Jun Geng\thanks{Supported in part by NNSF of China (No. 11971212).} \qquad Weisheng Niu\thanks{Supported by NNSF of China (No. 11971031).}
    }

\maketitle
\pagestyle{plain}
\begin{abstract}
 We investigate quantitative estimates in homogenization of the locally periodic parabolic operator with multiscales
 $$ \pa_t- \text{div} (A(x,t,x/\va,t/\kappa^2) \na ),\qquad \va>0,\, \kappa>0. $$ Under proper assumptions, we establish the full-scale interior and boundary Lipschitz estimates. These results are new even for the case $\kappa=\va$, and for the periodic operators $
 \pa_t-\text{div}(A(x/\va, t/\va^{\ell}) \na  ),$ $0<\va,\ell<\infty, $
of which the large-scale Lipschitz estimate down to $\va+\va^{\ell/2}$ was recently established by the first author and Shen in Arch. Ration. Mech. Anal.  236(1): 145--188 (2020). Due to the non-self-similar structure, the full-scale estimates do not follow directly from the large-scale estimates and the blow-up argument. As a byproduct, we also derive the convergence rates for the corresponding initial-Dirichlet problems, which extend the results in the aforementioned literature to more general settings.
 \end{abstract}



\section{Introduction}
 We investigate quantitative estimates in homogenization of the following second order parabolic operator in $\R^{d+1}$
\begin{align} \label{eq1}
 \pa_t +\mathfrak{L}_\varepsilon=\pa_t -\pa_{x_i}\big(A_{ij}^{\alpha \beta}(x,t,x/\va, t/\ka^2) \pa_{x_j}\big),
\end{align}
where $\va>0$, and $\kappa=\ka(\va)>0$ satisfies the assumption
\begin{align}\label{ratio}
  \kappa\rightarrow 0  \,\text{ as }\, \va\rightarrow 0, \quad\text{ and } \quad\lim_{\va\rightarrow 0} \frac{\kappa}{\va}=\rho.
\end{align}
 The coefficient  matrix $A(x,t,y,s)=(A_{ij}^{\alpha \beta}(x,t,y,s))$, $1\leq i,j\leq d, 1\leq \alpha, \beta\leq n$, is real, bounded measurable with
\begin{align}\label{cod1}
 \mu |\xi|^2\leq   A_{ij}^{\al\be}(x,t,y,s) \xi_i^\al\xi_j^\be , ~~ ~~\|A\|_{L^\infty(\R^{2d+2})}\leq\frac{1}{\mu}
\end{align}
for a.e. $(x,t),(y,s)\in \R^{d+1}$ and any $\xi=(\xi_i^\al)\in \R^{n\times d}$, where $\mu>0$.   We also assume that $A$ is $1$-periodic in $(y,s),$ i.e.,
\begin{align}\label{cod2}
A(x,t,y+z,  s+\tau)=A(x,t,y,s)
\end{align}
 for any $(z,\tau)\in \mathbb{Z}^{d+1}$ and a.e. $(x,t), (y,s)\in \R^{d+1},$
and $A$ satisfies the H\"{o}lder continuity condition in $(x,t)$: there exist $L>0$ and $\theta\in (0,1]$ such that
\begin{align}\label{cod3}
|A(x,t,y,s)-A(x',t',y,s)|\leq L\big\{|x-x'|+ |t-t'|^{1/2} \big\}^\theta,
\end{align}
for any $ x, x', y \in \R^d$ and $t, t', s  \in \R.$ Note that no smoothness assumption is made on $(y,s)$.

Quantitative homogenization of partial differential equations has witnessed a fast growth in the past decades. As the main issues of the quantitative theory, convergence rates and uniform regularity estimates have been widely studied for elliptic equations and systems in various contexts, see \cite{al87,al89,klsamar2012, klsjams2013, gloria2014regularity, armstrongan2016,  gloriajems2017,shenan2017, shenzhu2017} and a great number of references in \cite{shennote2018,armstrongbook2019}.
Quantitative homogenization of parabolic problems has also aroused great interest in recent years. As the prototype,
quantitative estimates in homogenization of the following operator
\begin{align}\label{eqk=2}
\pa_t -\text{div}\big(A(x/\va, t/\va^2) \na\big)
\end{align}
have been largely studied.  The uniform interior and boundary Lipschitz estimates were established in  \cite{geng2015uniform} and \cite{gengjde2020}, while the convergence rates have been studied in different contexts in \cite{gsjfa2017, xup2017,nxjde2019}.
    See also \cite{suslina2016, nxdcds2018,Armstrongapde-2018, gsapde2020} for more related results.

Very recently, quantitative estimates in periodic homogenization of the following parabolic operator  with non-self-similar scales
\begin{align}\label{eqk}
\pa_t -\text{div}\big(A(x/\va, t/\va^{\ell}) \na\big),\quad  0<\ell<\infty,
\end{align}
were studied by the first author and Shen in \cite{gsamar2020}.
Since the scaling of $\va$ in the coefficients $A(x/\va,t/\va^{\ell})$ does not agree with the intrinsic scaling of the second order parabolic operators.  Homogenization of the operator \eqref{eqk} is much more involved than \eqref{eqk=2}. One might get some insights from the qualitative theory in the celebrated monograph \cite{lions1978}, which shows that the temporal and spatial variables do not homogenize  simultaneously for $\ell\neq2$, and particularly the homogenized operator depends on $\ell$ in three cases: $0<\ell<2; \ell=2$; and $2<\ell<\infty$. By introducing a proper family of $\lm$-dependent equations, the authors in \cite{gsamar2020} developed an effective approach and conducted a systematic study on quantitative homogenization of \eqref{eqk}. In particular, they derived the interior and boundary Lipschitz estimates uniform down to the scale $\va+\va^{\ell/2}$. However, the uniform Lipschitz estimates in the small scale remain unknown. Due to the non-self-similar structure, the small-scale estimates do not follow directly from the large-scale estimates and the blow-up argument   as the classical elliptic or parabolic operators considered in \cite{al87,geng2015uniform,shenan2017,shennote2018}, and therefore require more efforts.

The primary aim of this paper is to establish the full-scale (both the large-scale and the small-scale) uniform interior and boundary Lipschitz estimates for the locally periodic parabolic operator  \eqref{eq1}, which, as a special case, provide the small-scale Lipschitz estimates  for the non-self-similar operator \eqref{eqk}.  Qualitative theory in homogenization of locally periodic parabolic operators has been widely studied, see \cite{lions1978,bench2004,Fold2006, wou2010, persson2012} and the large amount of references therein. However, the corresponding quantitative results are very limited. In \cite{past2007,past2010} the optimal convergence rates in homogenization of some locally periodic parabolic operators with time-independent coefficients were considered. In \cite{xaxive2020} the sharp convergence rate was obtained for the locally periodic operator \eqref{eq1} with $\kappa=\va$. To our best knowledge, very few results have been obtained on uniform regularity estimates in homogenization of locally periodic parabolic operators even for the operator \eqref{eq1} with $\kappa=\va$.

To present our main results, we use the following notations.
 For $ (x_0,t_0)\in \R^{d+1}$, let $ Q_r(x_0,t_0)=B(x_0,r) \times (t_0-r^2, t_0),$ where $B(x_0,r)=\{x\in \R^d: |x-x_0|<r\}$. Let $\Omega$ be a bounded domain in $\R^d$. For $x_0\in \pa \Omega$ and $t_0\in \R$, define the sets $ D_r(x_0,t_0)$ and $ \De_r(x_0,t_0)$ as follows, $$D_r(x_0,t_0)=(B(x_0,r) \cap \Omega) \times (t_0-r^2, t_0),\quad \De_r(x_0,t_0)=(B(x_0,r) \cap \pa\Omega) \times (t_0-r^2, t_0).$$
For $E\subseteq \R^{d+1}$, we use  $C^{1+\al}(E), 0<\al\leq1$, to denote the parabolic H\"{o}lder space with scale-invariant norm
\begin{align}
\| f\|_{C^{1+\al}(E)}=\|f\|_{L^\infty(E)}+r\|\na  f\|_{L^\infty(E)}+ r^{1+\al}\|\na  f\|_{C^{\al,\al/2}(E)}+r^{1+\al} \|f\|_{C_t^{(1+\al)/2}(E)},
\end{align}
where $\|\cdot\|_{C^{\al,\al/2}(E)}$ and $ \|\cdot\|_{C_t^\al(E)}$ are the parabolic H\"{o}lder seminorms defined as follows
\begin{align}\label{htseminorm}
\begin{split}
&\|u\|_{C^{\al,\al/2}(E)}=\sup_{\substack{(x,t), (y,s)\in E\\ (x,t)\neq (y,s)}} \frac{|u(x,t)-u(y,s)|}{(|x-y|+|t-s|^{1/2})^\al}, \\
&\|u\|_{C_t^\al(E)}=\sup_{\substack{(x,t), (x,s)\in E\\  t \neq  s} } \frac{|u(x,t)-u(x,s)|}{|t-s|^{\al}}.\end{split}
\end{align}

Our first two results provide the uniform interior and boundary Lipschitz estimates for the operator \eqref{eq1}.
 \begin{theorem}\label{thm1}
 Suppose  $A(x,t,y,s)$ satisfies \eqref{cod1}-\eqref{cod3}.  Assume that
$\pa_t u_\va +\mathfrak{L}_\va u_\va =F $ in $Q_1=Q_1(x_0,t_0)$ with $F\in L^p(Q_1)$ for some $p>d+2$.  Then for any  $0<\va+\ka\leq r<1$,
\begin{equation}\label{thm1-re1}
\Big(\fint_{Q_r} |\nabla u_\va|^2\Big)^{1/2}
\le C \Big\{ \Big(\fint_{Q_1} |\nabla u_\va|^2 \Big)^{1/2}
+ \Big(\fint_{Q_1} |F|^p \Big)^{1/p} \Big\},
\end{equation}
where $C$ depends only on $d$, $n$, $p$, $\mu$, and $(\theta,L)$ in \eqref{cod3}.

If, in addition,  \eqref{ratio} holds and  there exist $M>0$ and $\vartheta\in (0,1]$ such that
\begin{align}\label{cod33}
|A(x,t,y,s)-A(x',t',y',s')|\leq M \big\{|x-x'|+|y-y'|+ |t-t'|^{1/2}+|s-s'|^{1/2} \big\}^\vartheta
\end{align}
for any $ x,x',y,y' \in \R^d$ and $t,t', s,s' \in \R.$ Then for any $\va,\kappa>0$ and $0<r\leq 1$, we have
\begin{align}\label{thm1-re2}
 |\na u_\va(x_0,t_0)|\leq C \Big\{\Big(\fint_{Q_r}|\na u_\va|^2\Big)^{1/2} +r\Big(\fint_{Q_r}|F|^p\Big)^{1/p} \Big\},
\end{align}
where $C$ depends only on $d$, $n$, $\mu$, $p$, and $(\vartheta,M)$ in \eqref{cod33}, and $\rho$ as well as the convergence rate of $ \kappa/\va$ to $\rho$ in \eqref{ratio} .
\end{theorem}


\begin{theorem}\label{b-thm2}
Let $\Omega$ be a bounded $C^{1,\al} (0\!<\!\al\!<\!1)$ domain in $\R^d$.
Suppose  $A(x,t,y,s)$ satisfies \eqref{cod1}-\eqref{cod3}.  Assume that
 $\pa_t u_\va +\mathfrak{L}_\va u_\va =F$ in $D_1=D_1(x_0,t_0)$ with $F\in L^p(D_1)$ for some $p>d+2$, and $u_\va=g$  on $\De_1=\De_1(x_0,t_0)$ with $g\in C^{1+\al}(\De_1)$. Then for any $0<\va+\kappa\leq r<1$,
 \begin{equation}\label{b-thm2-re1}
\Big(\fint_{D_r } |\nabla u_\va|^2\Big )^{1/2}
\le C \Big\{ \Big(\fint_{D_1 } |\nabla u_\va|^2 \Big)^{1/2}
+  \| g\|_{C^{1+\al}(\Delta_1)}
+  \Big(\fint_{D_1 } |F|^p \Big)^{1/p} \Big\},
\end{equation}
where $C$ depends only on $d$, $n$, $p$, $\mu$, $\al$, $\Omega$,  and $(\theta,L)$ in \eqref{cod3}.

If in addition \eqref{ratio} holds and $A$ satisfies \eqref{cod33}. Then for any $\va,\kappa>0$ and $0<r\leq1$,
 \begin{align}\label{b-thm2-re2}
   |\na u_\va(x_0,t_0) |  \leq C \Big\{\Big(\fint_{D_r }\!|\na u_\va|^2\Big)^{1/2} \!+\! r\Big(\fint_{D_r }|F|^p\Big)^{1/p}\!+\!r^{-1}\|g\|_{C^{1+\al}(\De_r )}\Big\},
\end{align}
where $C$ depends only on $d$, $n$, $\mu$, $p$, $\al$, $\Omega$, and $(\vartheta,M)$ in \eqref{cod33}, and $\rho$ as well as the convergence rate of $\kappa/\va$ to $\rho$ in \eqref{ratio}.
\end{theorem}

The estimates in \eqref{thm1-re1} and \eqref{b-thm2-re1} provide the large-scale uniform interior and boundary Lipschitz estimates for the operator $ \pa_t +\mathfrak{L}_\varepsilon$, which are totally induced by the homogenization process. The scale $\va+\kappa$, which is composed of the scales of temporal and spatial oscillations, is more or less optimal. On the other hand, the estimates \eqref{thm1-re2} and \eqref{b-thm2-re2} give the uniform interior and boundary Lipschitz estimates in the small scale, which is attributed to the smoothness of the coefficients. These estimates, on the one hand, provide the full-scale Lipschitz estimates for the locally periodic operator $\pa_t-\text{div}(A(x,t,x/\va,t/\va^2)\na )$. On the other hand, they imply the uniform interior and boundary Lipschitz estimates in the small scale for the non-self-similar operator \eqref{eqk}, which were left out in \cite{gsamar2020}.

The proof of \eqref{thm1-re1} and \eqref{b-thm2-re1} is motivated by \cite{gsamar2020} and \cite{Niu-Xu-2020}. First, we follow the idea of \cite{gsamar2020} to introduce a family of $\lm$-dependent operators
\begin{align*}
\pa_t +\mathcal{L}^\lm_\va =\pa_t-\text{div} (A^\lm(x,t,x/\va,t/\va^2) \na )
\end{align*} with
$A^\lm(x,t,y,s)=A(x,t,y,s/\lm)$. This in some sense is a kind of scale-reduction argument, as formally the resulting operator has only scale $\va$, and the scale $\kappa$ has gone. In particular, for fixed $0<\lm<\infty$ the resulting operator is self-similar as the one in \eqref{eqk=2}.
Then by adapting the approach developed for studying the large-scale regularity in homogenization of elliptic problems in \cite{armstrongan2016, shenan2017}, we establish the large-scale interior and boundary Lipschitz estimates for $\pa_t+\mathcal{L}_\va^\lm$ with bounding constant independent of $\lm$. As a result, the large-scale estimates for $\pa_t+\mathfrak{L}_\va$ follows by setting $\lm=\kappa^2/\va^2$.   To carry out the plan above, we use some ideas in \cite{Niu-Xu-2020} to deal with the locally periodic operators. In particular, we derive some estimates with sharp bounding constants depending explicitly on $\|\na_x A\|_\infty$ and $\|\pa_t A\|_\infty$, which allow us to quantify the smooth approximation of the coefficients, and  are essential in the derivation of the large-scale estimates of $\pa_t+\mathcal{L}_\va^\lm$. See Sections 2, 3 and 4 for the details.

Armed with \eqref{thm1-re1} and \eqref{b-thm2-re1}, we then perform the blow-up analysis to establish the full-scale interior and boundary Lipschitz estimates for the operators with only spatial or temporal oscillations, i.e.,
\begin{align}\label{ilva}
 \pa_t-\text{div}(A(x,t,x/\va)\na ),  \quad\text{and}\quad
 \pa_t-\text{div}(A(x,t,t/\va^2)\na ).
\end{align}
Finally, by performing proper rescaling argument according to the value of $\rho$, we prove the desired estimates \eqref{thm1-re2} and \eqref{b-thm2-re2}.

As a byproduct of the above process, we also derive the convergence rates for the initial-Dirichlet problem
\begin{align}\label{eq2}
\pa_tu_\va +\mathfrak{L}_\varepsilon u_\va=F   \,\, \,\,\text{in } \Omega_T, \quad\quad
u_\va   =g   \,\, \,\,\text{on } \partial_p \Omega_T,
\end{align}
where $\Omega$ is a bounded domain in $\R^d$, $\Omega_T=\Omega\times (0,T),$ and  $\pa_p \Omega_T$ is the parabolic boundary of $\Omega_T$.

\begin{theorem}\label{cothm}
Let $\Omega$ be a bounded $C^{1, 1}$ domain in $\mathbb{R}^d$. Suppose \eqref{ratio} holds, and $A$ satisfies \eqref{cod1}, \eqref{cod2} and \eqref{cod3} with $\theta=1$.
Moreover,  assume that $\|\na^2_y A\|_{L^\infty(\R^{2d+2})}<\infty$ if $\rho=0$, and $\|\pa_s A\|_{L^\infty(\R^{2d+2})}<\infty$ if $\rho=\infty$.
Let $u_\va$ be a weak solution to \eqref{eq2} and $u_0$ the solution to the
homogenized problem.
Then
\begin{align}
\begin{split}\label{cothmre}
\| u_\va  -u_0\|_{L^2(\Omega_T)}
&\leq
 \big\{ \| u_0\|_{L^2(0, T; H^2(\Omega))}
+\|\partial_t u_0\|_{L^2(\Omega_T)} \big\}\\
&\quad \times \left\{
 \aligned
& C_1  \left\{ \kappa + \va + (\va \kappa^{-1})^2\right\}      \quad\text{ if }  \rho =\infty,\\
 &C_2 \left\{ \kappa +\va + \rho^{-2}  \big|   ( \kappa \va^{-1} )^2  -\rho^2   \big| \right\}  \quad \text{ if }  0< \rho < \infty,\\
 &C_3 \left\{ \kappa +\va + (\kappa \va^{-1})^2 \right\}     \quad \text{ if }  \rho=0,
 \endaligned
 \right.\end{split}
\end{align}
where $C_1$ depends only on $d$, $n$, $\mu$, $L$,  $T$, $\Omega,$ and  $\|\pa_s A\|_{L^\infty(\R^{2d+2})}$, $C_2$ depends only on $d$, $n$, $\mu$, $L$, $T,$ and  $\Omega$,
and  $C_3$ depends only on  $d$, $n$, $\mu$, $L$,  $T$, $\Omega,$ and  $ \|\na^2_y A\|_{L^\infty(\R^{2d+2})}$.
\end{theorem}

There are three terms in the convergence rate \eqref{cothmre}. The first two terms $\kappa,\va$ represent respectively the scales of the space and time variables. The last term might be understood as the product of the competition in homogenization between the temporal and spatial variables. Note that these two variables do not homogenize simultaneously if $\kappa\neq c \va$. Finally, we remark that
for the case $\kappa=\va$ the estimate \eqref{cothmre} gives the sharp $O(\va)$-order convergence rate for the locally periodic parabolic operator $\pa_t-\text{div} (A(x,t,x/\va,t/\va^2) \na)$, which has recently been studied in \cite{xaxive2020}. And for the case $\kappa=\va^{\ell/2}$, \eqref{cothmre} reduces to
\begin{align*}
\| u_\va  -u_0\|_{L^2(\Omega_T)}
&\leq
 C \big\{ \| u_0\|_{L^2(0, T; H^2(\Omega))}
+\|\partial_t u_0\|_{L^2(\Omega_T)} \big\} \cdot\left\{
 \aligned
&  \va^{\ell/2} + \va^{2-\ell}
    \text{ if } 0<\ell< 2,\\
   &  \va
 \quad \quad \quad\quad   \,\, \text{ if } \ell=2 ,\\
& \va+\va^{\ell-2}\quad  \,\,  \text{ if } 2<\ell<\infty,
 \endaligned
 \right.
\end{align*}
 which extends Theorem 1.3 in \cite{gsamar2020} to the locally periodic case.

The remaining parts of the paper is arranged as follows. In Section 2, we provide the qualitative homogenization of the locally periodic operator $\pa_t +\mathcal{L}^\lm_\va$. Particularly,  the estimates of the correctors and flux correctors, as well as some useful estimates of the smoothing operators are presented. In Section 3, we perform the two-scale expansion to approximate the solution of $\pa_t u_\va+\mathcal{L}^\lm_\va u_\va=F$ by solutions of the homogenized problem $\pa_t u_0+\mathcal{L}^\lm_0 u_0=F$, based on which, in Section 4  we establish the uniform large-scale interior Lipschitz estimate for the operator  $\pa_t+\mathfrak{L}_\va$. Moreover, we prove the full-scale Lipschitz estimates for the operators with only spatial and temporal oscillations in \eqref{ilva}.  Section 5 is devoted to the uniform boundary Lipschitz estimates for the operators in \eqref{eq1} and \eqref{ilva}. Armed with the results above, we complete the proof of Theorems \ref{thm1} and \ref{b-thm2} in Section 6. Finally, in Section 7 we discuss the convergence rates and provide the proof of Theorem \ref{cothm}.

 Throughout the paper, we use $\fint_E u$ to denote the $L^1$ average of $u$ over the set $E$, i.e.,
$\fint_E u=\frac{1}{|E|}\int_E u $. To simplify the notation, we assume $n=1$ hereafter. Yet, all our analysis extends directly to the case $n>1$ (parabolic systems), since no particular result affiliate to the scalar case is ever used.

\section{Preliminaries}

\subsection{Homogenization of $\pa_t +\mathcal{L}^\lm _\va$}
We investigate qualitative homogenization for the operator $\pa_t +\mathfrak{L}_\va$ in \eqref{eq1}.
 Consider the operator  \begin{align}\label{Llm}
 \pa_t +\mathcal{L}^\lm_\va =\pa_t -\text{div} ( A^\lm (x,t,x/\va, t/ \va^2) \na ),\end{align}
where $A^\lm =A^\lm(x,t,y,s) =A(x,t,y,s/\lm )$ for any $x,y\in \R^{d}$ and $t,s\in \R$. Assume that $A=A(x,t,y,s)$  satisfies \eqref{cod1} and \eqref{cod2}.
Then the matrix $A^\lm$ satisfies \eqref{cod1} with the same constant, and is $(1,\lm)$-periodic in $(y, s)$, i.e.,
\begin{align}
A^\lm (x,t,y,s)= A^\lm(x,t,y+z, s+\lm \tau)
\end{align}
 for any  $(z,\tau)\in \mathbb{Z}^{d+1} \text{ and }\, a.e.\, x,y\in \R^d, s,t\in \mathbb{R}.$

For  $1\leq i,j\leq d$, let $\chi^\lm (x,t,y,s)= (\chi^\lm _{j}(x,t,y,s))$ be the weak solution of the following cell problem
 \begin{equation}\label{chi}
\begin{cases}
 \pa_s \chi_{j}^\lm  (x,t,y,s)   -    \pa_{y_i} ( A^\lm _{ik}   \pa_{y_k}  \chi^\lm _{j} (x,t,y,s)  )
 =    \pa_{y_i}  A^\lm _{ij}(x,t,y,s)
 ~~    \text{ in } \R^{d+1},\\
 \chi^\lm _{j}  (x,t,y,s)\quad\text{is } (1,\lm )\text{-periodic in } (y,s), \\
 \int_0^\lm\!  \int_{\mathbb{T}^{d}} \chi^\lm _{j} (x,t,y,s)\, dy\,ds=0, \quad 1\leq j\leq d.
 \end{cases}
\end{equation}
Standard energy estimate  and Poincar\'{e}'s inequality imply that
\begin{align}& \fint_0^\lm \!\!\int_{\mathbb{T}^{d}} |\na_y \chi^\lm (x,t,y,s)|^2dyds+ \fint_0^\lm \!\! \int_{\mathbb{T}^{d}} |\chi^\lm (x,t,y,s)|^2dyds \leq C. \label{lecore0}\end{align}
Define
\begin{align}\label{alamhat}
\widehat{A^\lm }(x,t)=\fint_0^\lm \!\!\!\int_{\mathbb{T}^{d}} \big(A^\lm (x,t,y,s) +A^\lm (x,t,y,s) \na_{y}\chi^\lm (x,t,y,s)\big )dyds.
\end{align}
In view of \eqref{lecore0}, we know that $\|\widehat{A^\lm } \|_{\infty}=\|\widehat{A^\lm }(x,t)\|_{L^\infty(\R^{d+1})}\leq C $ for some $C$ depending only on $d,\mu$.
Moreover, it is not difficult to prove that
\begin{align}\label{almecon}
 \mu |\xi|^2\leq \widehat{A^\lm_{ij}}(x,t) \xi_i\xi_j \quad \text{ for any } \xi \in \R^d \text{ and  } a.e.~(x,t) \in \R^{d+1}.
\end{align}
Thanks to \cite{lions1978}, for each fixed $\lm>0$ the homogenized operator of $\pa_t +\mathcal{L}^\lm _\va$  is given by
\begin{align}\label{nonl0}
\pa_t+\mathcal{L}^\lm _0=\pa_t-\text{div} (\widehat{A^\lm }(x,t) \na ).
\end{align}

\begin{lemma}\label{lemma-2.2}Suppose $A=A(x,t,y,s) $  satisfies \eqref{cod1} and \eqref{cod2}.
Let $\chi^\lambda$ be given by (\ref{chi}).
Then there exists $q>2$, depending on $d$ and $\mu$, such that
\begin{equation}\label{2.00}
\left(\fint_0^\lambda \!\!\!\int_{\mathbb{T}^d}
|\nabla_y \chi^\lambda(x,t,y,s)|^q\, dy ds \right)^{1/q}
\le C
\end{equation} for a.e. $(x,t)\in \R^{d+1}$,
where $C$ depends only on $d$ and $\mu$.
\end{lemma}
\begin{proof}
The desired estimate is a consequence of the Meyers-type estimates for parabolic systems (see e.g. \cite[Appendix]{Armstrongapde-2018}). We refer \cite{gsamar2020} for the detailed proof.
\end{proof}


\begin{lemma}\label{le-co}
Assume that $A$ satisfies conditions \eqref{cod1}, \eqref{cod2}, and \eqref{cod3} for some $\theta\in (0,1]$ and $L>0.$ Then
\begin{align}
& \fint_0^\lm\!\!  \int_{\mathbb{T}^{d}} |\na_y\chi^\lm (x,t,y,s)\!-\!\na_y \chi^\lm (x',t',y,s)|^2dyds  \leq C L^2 (|x-x'|+|t-t'|^{1/2})^{2\theta}, \label{lecore1}\\
&|\widehat{A^\lm }(x,t)-\widehat{A^\lm }(x',t')|\leq  C L (|x-x'|+|t-t'|^{1/2})^\theta \label{lecore2}
\end{align}
for any    $x,x'\in \R^{d}$ and $t,t'\in \R$, where $C$ depends only on $d$ and $\mu$.
\end{lemma}
\begin{proof} Note that if $A$ satisfies \eqref{cod3}, $A^\lm$ satisfies \eqref{cod3} with the same constant.
The estimate \eqref{lecore1} follows from \eqref{cod3} and the standard energy estimates for
\begin{align}
&\pa_s \big(\chi^\lm_j(x,t,y,s)-\chi^\lm_j(x',t',y,s)\big) - \text{div}_y \big\{A^\lm(x,t,y,s)\na_y \big(\chi^\lm_j(x,t,y,s)-\chi^\lm_j(x',t',y,s) \big)\big\}\nonumber\\
&= \text{div}_y \big\{ \big(A^\lm(x,t,y,s)-A^\lm(x',t',y,s)\big) \na_y(  \chi^\lm_j(x',t',y,s) + y_j) \big\},\nonumber
\end{align}
while \eqref{lecore2} is a direct consequence of the definition of $\widehat{A^\lm}(x,t)$ and \eqref{lecore1}.
\end{proof}

To define the homogenized operator of \eqref{eq1}, for $1\leq j\leq d $ we introduce the correctors $\chi^\infty=(\chi^\infty_j(x,t,y,s))$ and $\chi^0=(\chi^0_j(x,t,y,s))$ given respectively  by
\begin{equation} \label{chiinfty}
\begin{cases}
-\text{\rm div}_y \big( A(x,t,y,s)\nabla_y \chi_j^\infty(x,t,y,s)) =\text{\rm div}_y(A(x,t,y,s)\nabla_y y_j) \quad \text{ in } \mathbb{R}^{d},\\
\chi_j^\infty(x,t,y,s) \text{ is 1-periodic in } (y, s),\\
 \int_{\mathbb{T}^d} \chi_j^\infty (x,t,y, s)\, dy=0,
\end{cases}
\end{equation}
and
\begin{equation}\label{chi0}
\begin{cases}
  -\text{\rm div}_y \left( \overline{A}(x,t,y)\nabla_y \chi_j^0(x,t,y) \right)=\text{\rm div}_y \left(\overline{A}(x,t,y)\nabla_y y_j\right) \quad \text{ in } \mathbb{R}^d,\\
 \chi_j^0(x,t,y) \text{ is 1-periodic in } y,\\
 \int_{\mathbb{T}^d} \chi_j^0(x,t,y)\, dy =0,
\end{cases}
\end{equation}
 where $
 \overline{A}(x,t,y) =\int_0^1 A(x,t,y,s) ds.
 $
By standard energy estimates and Poincar\'{e}'s inequality,
\begin{align}
\int_{\mathbb{T}^d} \big(|\na_y\chi_j^\infty (x,t,y, s)|^2+|\chi_j^\infty (x,t,y, s)|^2\big)dy\leq C, \label{esti-chiinfty}\\
\int_{\mathbb{T}^d} \big(|\na_y\chi_j^0 (x,t,y)|^2+|\chi_j^0 (x,t,y)|^2\big)dy\leq C\label{esti-chi0}
\end{align}
for a.e. $(x,t) \in \R^{d+1}$ and $s\in \R.$
Define
\begin{align}
&\widehat{A^\infty}(x,t)=\int_0^1\!\!\!\int_{\mathbb{T}^d} (A(x,t,y,s)+A(x,t,y,s)\na_y \chi^\infty(x,t,y,s)) dyds,\label{Ainftyhat}\\
&\widehat{A^0}(x,t)
 =\int_0^1  \!\!\! \int_{\mathbb{T}^d}
\left( A(x,t,y,s) + A(x,t,y,s) \nabla_y \chi^0(x,t,y) \right) dyds .\label{A0hat}
\end{align}
In view of \eqref{esti-chiinfty} and \eqref{esti-chi0}, we know that
  $\|\widehat{A^\infty}\|_{\infty}\leq  C$ and $\|\widehat{A^0}\|_{\infty}\leq C$, where $C$ depends only on $d,\mu$; and similar to \eqref{almecon}, one has
\begin{align}
\widehat{A^\infty_{ij}}(x,t)\xi_i\xi_j \geq \mu |\xi|^2 \quad\text{and}\quad
\widehat{A^0_{ij}}(x,t)\xi_i\xi_j \geq \mu |\xi|^2
\end{align}
 for any  $\xi \in \R^d$ and  a.e.  $(x,t) \in \R^{d+1}$.

 \begin{lemma}\label{lemconA}
Suppose $A=A(x,t,y,s)$  satisfies \eqref{cod1} and \eqref{cod2}.
  Let $\widehat{A^\lambda}$ be defined as in \eqref{alamhat}.
  Then \begin{align}
 &\|\widehat{A^\lambda }- \widehat{A^\infty} \|_{\infty}
\leq C \lm^{-1}\|\pa_s A\|_{\infty}  \, \,\text{ for } 1\le \lm<\infty, \text{ if } \|\pa_s A\|_{\infty} < \infty, \label{lemconAre1}\\
&\|\widehat{A^{\lambda_1} }-\widehat{A^{\lambda_2} }  \|_{\infty}\leq  C   \big|1- \lm_2 \lm_1^{-1}\big|    \, \,\text{ for  }
 0< \lambda_1, \lambda_2<\infty,  \label{lemconAre2} \\
&\|\widehat{A^\lambda }-\widehat{A^0} \|_{\infty}
 \leq
   C\lm \big(\|\na^2_y A\|_{\infty} + \| \na_y A\|^2_{\infty}\big)   \, \,\text{ for  }  0< \lambda\leq 1,  \text{ if } \|\na^2_y A\|_{\infty}<\infty, \label{lemconAre3}
 \end{align}
 where $C$ depends only on $d$ and $\mu$.
 \end{lemma}
\begin{proof} The results in \eqref{lemconAre1} and \eqref{lemconAre3} correspond to Theorems 2.3 and 2.5 in \cite{gsamar2020}, and the proofs are almost the same. We provide the  proof of \eqref{lemconAre2} for the convenience of the reader.
Note that
\begin{align}\label{pro-prelem1-4}
\begin{split}
&| \widehat{A^{\lambda_1}}(x,t)-\widehat{A^{\lm_2}}(x,t)|\\
&= \int_{\mathbb{T}^{d+1}}
A (x,t,y, s) \nabla_y  \{ \chi^{\lambda_1} (x,t,y, \lambda_1 s)-\chi^{\lm_2}(x,t,y, \lm_2s) \} dyds\\
&\le C \Big( \int_{\mathbb{T}^{d+1}}
\big| \nabla_y \{ \chi^{\lm_1} (x,t,y, \lambda_1 s)-\chi^{\lm_2} (x,t,y,\lm_2 s) \} \big|^2   dyds\Big)^{1/2}.
\end{split}
\end{align}
By the definition of $\chi^\lm$,
\begin{align}\label{pro-prelem1-5}
\begin{split}
 &\text{\rm div}_y \big\{ A(x,t,y, s) \nabla_y \big( \chi_j^{\lambda_1} (x,t,y, \lambda_1 s) -\chi_j^{\lm_2} (x,t,y, \lm_2s) \big) \big\}\\
 &=\Big(\frac{1}{\lambda_1}-\frac{1}{\lambda_2}\Big)\pa_s \big(\chi^{\lm_2}(x,t,y,\lm_2s)\big)+\frac{1}{\lambda_1} \big\{\pa_s \big( \chi^{\lm_1}(x,t,y,\lm_1s)\big) -\pa_s \big(\chi^{\lm_2}(x,t,y,\lm_2s)\big)\big\}
\end{split}
\end{align}
in $\mathbb{T}^{d+1}.$
By \eqref{cod1} and the fact that
\begin{align*}
\int_{\mathbb{T}^{d+1}}
 \pa_s
\big( \chi_j^{\lambda_1} (x,t,y, \lambda_1 s) -\chi_j^{\lm_2}(x,t,y, \lm_2 s)  \big) \cdot \big( \chi_j^{\lm_1} (x,t,y, \lambda_1 s) -\chi_j^{\lm_2}  (x,t,y, \lm_2 s) \big)\, dyds =0,
\end{align*}
we deduce that
\begin{align}\label{pro-prelem1-00}
\begin{split}
 & \mu \int_{\mathbb{T}^{d+1}}
\big|\nabla_y \big(\chi_j^{\lambda_1} (x,t,y, \lambda_1 s) -\chi_j^{\lm_2} (x,t,y, \lm_2 s) \big) \big|^2\, dyds\\
&\leq \Big|\Big(\frac{1}{\lambda_1}-\frac{1}{\lambda_2}\Big)
\int_{\mathbb{T}^{d+1}}
\frac{\partial}{\partial s}
 \chi_j^{\lm_2} (x,t,y, \lm_2 s)  \cdot \big(\chi_j^{\lambda_1} (x,t,y, \lambda_1 s) -\chi^{\lm_2} _j (x,t,y, \lm_2 s) \big)  \, dyds\Big|.
 \end{split}
\end{align}
 Since
$$
\int_{\mathbb{T}^d}
\chi_j^{\lambda_i} (x,t,y, \lambda_i s)\, dy
=0, \qquad i=1,2.
$$
By Poincar\'e's inequality and the Cauchy inequality,
\begin{align*}
  &\Big(\int_{\mathbb{T}^{d+1}}
\big|\nabla_y \big( \chi_j^{\lambda_1} (x,t,y, \lambda_1 s) -\chi_j^{\lm_2} (x,t,y,\lm_2 s) \big) \big|^2\, dyds\Big)^{1/2} \\
&\le C \big| \lambda_1^{-1}- \lambda_2^{-1}\big| \Big( \int_0^1
\|\partial_s  \chi^{\lm_2}(x,t,y,\lm_2 s) \|^2_{H^{-1}_{per}(\mathbb{T}^d)} ds \Big)^{1/2} .
\end{align*}
In view of \eqref{pro-prelem1-4}, this implies that
\begin{equation}\label{pro-prelem1-6}
|\widehat{A^{\lambda_1}}(x,t) -\widehat{A^{\lm_2}}(x,t)|
\le C \big|\lm_2^{-1} -\lambda_1^{-1}\big| \|\partial_s  \chi^{\lm_2}(x,t) \|_{L^2(0,\lm_2; H^{-1}_{per}(\mathbb{T}^d))},
\end{equation}
for any $(x,t)\in \R$, where $C$ depends only on $d$ and $\mu$. Thanks to \eqref{chi}, we have
\begin{align}\label{pro-prelem1-7}
\begin{split}
\|\partial_s  \chi^{\lm_2}(x,t) \|_{L^2(0,\lm_2; H^{-1}_{per}(\mathbb{T}^d))}
&\le C \|A^{\lm_2}(x,t)\|_{L^2(0,\lm_2; L^2(\mathbb{T}^d))}\\
&\quad+ C\|A^{\lm_2}(x,t) \na_y \chi^{\lm_2}(x,t)\|_{L^2(0,\lm_2; L^2(\mathbb{T}^d))},
\end{split}
\end{align}
which, combined with \eqref{lecore0}
and \eqref{pro-prelem1-6}, gives \eqref{lemconAre2}.

\end{proof}

Define
 \begin{align}\label{ahat}
\widehat{A}(x,t)=\left\{
 \aligned
& \widehat{A^0} (x,t)& \quad & \text{ if }  \rho=0,\\
& \widehat{A^\rho}(x,t) & \quad & \text{ if }   0<\rho<\infty,\\
& \widehat{A^\infty}(x,t) & \quad & \text{ if } \rho=\infty.
 \endaligned
 \right.
 \end{align}
\begin{lemma}\label{ahattoa}  Suppose  $A=A(x,t,y,s)$  satisfies \eqref{cod1} and \eqref{cod2}.
Then $\|\widehat{A^\lambda} -\widehat{A}\|_{\infty} \rightarrow 0$ as $\lambda \to \rho$.
\end{lemma}
\begin{proof}
The result for the case $0< \rho <\infty$ is a direct consequence of \eqref{lemconAre2}, while the results for the cases $\rho=0$ and $\rho=\infty$ can be proved by smooth approximation as Theorems 2.3 and 2.5 in \cite{gsamar2020}.
\end{proof}

The following theorem provides the homogenized operator of $\pa_t+\mathfrak{L}_\va$ in \eqref{eq1}.
\begin{theorem}  Suppose \eqref{ratio} holds and $A=A(x,t,y,s) \in C(\R^{d+1}; L^\infty (\R^{d+1}))$  satisfies \eqref{cod1} and \eqref{cod2}.
Let $\Omega\subseteq \R^d$ be a bounded Lipschitz domain, and $u_\va$  a weak solution to $\pa_t u_\va +\mathfrak{L}_\va u_\va=F$ in $\Omega\times (T_1,T_2),$ where $F\in L^2(T_1,T_2; W^{-1,2}(\Omega)).$  Then as $\va$ tends to zero,   $u_\va$ converges to $u_0$ weakly in $L^2(T_1,T_2; W^{1,2}(\Omega))$ and strongly in $L^2(\Omega\times (T_1,T_2))$, where $u_0$ is the solution  to
\begin{align*}
\pa_tu_0-\text{div} (\widehat{A}(x,t) \na u_0)=F  \quad\text{in }\Omega\times (T_1,T_2),
\end{align*} with $\widehat{A}$ being given by \eqref{ahat}.
\end{theorem}
This theorem has been proved in \cite{persson2012} by using the multi-scale convergence method. Based on Lemma \ref{ahattoa}, one can also prove the theorem by using Tartar's test function method. Yet we shall not provide the details, as our analysis does not rely on the qualitative result above, and moreover by proper approximation our quantitative estimate \eqref{cothmre} provides an alternative proof for it.

\subsection{Flux correctors}
For $1\leq \bar{i}\leq d+1$, define
\begin{align} \label{defb}
B^\lm_{\bar{i}j} (x,t,y,s) =\left\{
  \begin{array}{ll}
   A^\lm_{\bar{i}j}  +  A^\lm_{\bar{i}k}   \pa_{y_k}  \chi^\lm_{j}   -\widehat{A^\lm_{\bar{i}j}} , &\textrm{ if } 1 \leq \bar{i}\leq d,\\
   -\chi^\lm_{j} , &\textrm{ if } \bar{i}=d+1.
  \end{array}
\right.
\end{align}
In view of \eqref{chi} and \eqref{lecore0}, we have
\begin{align}
& \pa_{y_i}B^\lm_{ij} (x,t,y,s)+\pa_s  B^\lm_{(d+1)j} (x,t,y,s)=0.\label{divb=0}
\end{align}

\begin{lemma}\label{le-fcon}
Let $1\leq \bar{i}, \bar{j}\leq d+1$ and $1\leq i,j,k \leq d$. There exist functions $\mathfrak{B}^\lm_{\bar{i}\bar{j} k} (x,t,y, s)$ in $\R^{d+1} \times \R^{d+1}$, which are  $(1,\lm)$-periodic in $(y,s),$ such that
\begin{align}
&\mathfrak{B}^\lm_{\bar{i}\bar{j} k}  =- \mathfrak{B}^\lm_{\bar{j}\bar{i}k} ,\quad \pa_{y_{i}} \mathfrak{B}^\lm_{i\bar{j} k} (x,t,y,s)+ \pa_s \mathfrak{B}^\lm_{(d+1) \bar{j} k} (x,t,y,s)= B^\lm_{\bar{j} k} (x,t,y,s),\label{huab1}\\
& \fint_0^\lm\!\!\!\int_{\mathbb{T}^{d}} \big(|\mathfrak{B}^\lm_{ij k}(x,t,y,s) |^2 +|\na_y \mathfrak{B}^\lm_{i(d+1)k} (x,t,y,s) |^2\big) dyds \leq C,\label{huab2}\\
 &\fint_0^\lm\!\!\!\int_{\mathbb{T}^{d}}  |  \mathfrak{B}^\lm_{i(d+1)k} (x,t,y,s) |^2 dyds \leq  C (1+\lm)^2.\label{huab3}
\end{align}
Moreover, if $
\|\mathcal{D}  A^\lm\|_{\infty}=\|\mathcal{D} A^\lm\|_{L^\infty(\R^{2d+2})} < \infty,
$ for  $\mathcal{D}=\na_x$ or $\pa_t$.  Then
\begin{align} \label{eshuab}
\begin{split}
&\fint_0^\lm\!\!\!\int_{\mathbb{T}^{d}} \big(|\mathcal{D} \na_y\mathfrak{B}^\lm_{i(d+1) k}(x,t,y,s) |^2+|\mathcal{D}\mathfrak{B}^\lm_{ij k}(x,t,y,s) |^2\big) dyds \leq C \| \mathcal{D} A^\lm\|^2_{\infty},  \\
&\fint_0^\lm\!\!\!\int_{\mathbb{T}^{d}}  |  \mathcal{D}\mathfrak{B}^\lm_{i(d+1)k} (x,t,y,s) |^2 dyds \leq  C (1+\lm)^2 \|\mathcal{D}A^\lm\|^2_{\infty}
\end{split}
\end{align} for any $(x,t)\in \R^{d+1}$, where  $C$ depends only on $d,\mu.$
\end{lemma}
\begin{proof}
The construction of functions $\mathfrak{B}^\lm_{\bar{i}\bar{j} k}$ is completely the same as the case $A^\lm=A^\lm(y,s)$ in \cite{gsjfa2017}. The key idea is to solve
\begin{align*}
\begin{split}
&(\Delta_{y}+\pa_s^2) f^\lm_{\bar{i}k}(x,t,y,s)=B^\lm_{\bar{i}k}(x,t,y,s)  \quad \text{in}\quad \mathbb{T }^{d}\times (0,\lm) ,\\
&f^\lm_{\bar{i}k} (x,t,y,s) \,\,\text{ is  } (1,\lm)\text{-periodic in }\,  (y, s)
\end{split}
\end{align*}
for fixed $(x,t)\in \R^{d+1}$, and to set for  $ 1\leq i,j,k\leq d,$
\begin{equation}\label{plefcon-1}
 \begin{cases}
\mathfrak{B}^\lm_{ijk} =\pa_{y_i} f^\lm_{jk}-\pa_{y_j} f^\lm_{ik},     \\
\mathfrak{B}^\lm_{(d+1)jk} =\pa_{s} f^\lm_{jk}-\pa_{y_j} f^\lm_{(d+1) k},  \text{ and } \mathfrak{B}^\lm_{(d+1)jk} =- \mathfrak{B}^\lm_{j(d+1)k}.
\end{cases}
 \end{equation}
It is not difficult to verify that $\mathfrak{B}^\lm_{\bar{i}\bar{j} k}$ satisfy \eqref{huab1}.  To prove \eqref{huab2} and \eqref{huab3}, we use the Fourier series to write for fixed $i,j$,
\begin{align*}
B^\lm_{ij}(x,t,y,s)=\sum_{\substack{ n\in \mathbb{Z}^d, m\in \mathbb{Z}\\ (n, m)\neq (0, 0) }} b_{n,m}(x,t) e^{-2\pi  \textbf{i}  n\cdot y -2\pi \textbf{i} m s\lambda^{-1}}.
\end{align*}
Then
$$
f_{ij}^\lambda (x,t,y, s)=-\frac{1}{4\pi^2}
\sum_{\substack{ n\in \mathbb{Z}^d, m\in \mathbb{Z}\\ (n, m)\neq (0, 0) }}
\frac{ b_{n, m}(x,t)}{ |n|^2 +|m|^2 \lambda^{-2}}
 e^{-2\pi \textbf{i} n\cdot y -2\pi \textbf{i} m s\lambda^{-1}}.
$$
By Parseval's Theorem, it is not difficult to see that
\begin{align}\label{plefcon-2}
  \fint_0^\lambda \!\!\!\!\int_{\mathbb{T}^d}
 |\partial_s f_{ij}^\lambda(x,t,y,s)|^2dyds
 \le C \lambda^2, \end{align}
and \begin{align}
\begin{split}\label{plefcon-3}
 \fint_0^\lambda\! \!\!\!\int_{\mathbb{T}^d} &
\Big(  |\nabla_y f_{ij}^\lambda|^2
+ |\nabla_y^2 f_{ij}^\lambda|^2
+ |\partial_s^2 f_{ij}^\lambda|^2
+|\nabla_y \partial_s f_{ij}^{\lambda}|^2 \Big)dyds  \\
 & \le C \sum_{n, m}  |b_{n, m}|^2
  = C \fint_0^\lambda \!\!\!\!\int_{\mathbb{T}^d}
 |B_{ij}^\lambda(x,t,y,s)|^2dyds
   \le C\\
   \end{split}
 \end{align} for any $(x,t)\in \R^{d+1}$,
 where $C$ depends only on $d$ and $\mu$. Likewise, by \eqref{lecore0} we obtain that
 \begin{align}\label{plefcon-4}
 \fint_0^\lambda \!\!\!\int_{\mathbb{T}^d}
\Big(  |\nabla_y f_{(d+1)j}^\lambda|^2
+ |\nabla_y^2 f_{(d+1)j}^\lambda|^2
+ |\partial_s^2 f_{(d+1)j}^\lambda|^2
+|\nabla_y \partial_s f_{(d+1)j}^{\lambda}|^2 \Big)dyds
\le C
 \end{align} for any $(x,t)\in \R^{d+1}$.
In view of \eqref{plefcon-1}, one derives \eqref{huab2} and \eqref{huab3} from \eqref{plefcon-2}-\eqref{plefcon-4} immediately.

We now turn to \eqref{eshuab}.
Note that if $\|\mathcal{D}  A^\lm\|_{\infty} < \infty $ for $ \mathcal{D}=\na_x$ or $\pa_t$,
 instead of \eqref{lecore1} and \eqref{lecore2}  we have
\begin{align}\label{lecore3}
& \fint_0^\lm\!\!\!\int_{\mathbb{T}^{d}} \big(|\mathcal{D}\na_y\chi^\lm(x,t,y,s)|^2+ |\mathcal{D}\chi^\lm(x,t,y,s)|^2\big)dyds \leq C \|\mathcal{D} A^\lm\|^2_{\infty},
\end{align} where $C$ depends only on $d,\mu$. By the definition of $\widehat{A^\lm}$, we have
\begin{align}\label{lecore3'}
 \|\mathcal{D}\widehat{A^\lm}\|_{\infty} \leq C \|\mathcal{D} A^\lm\|_{\infty},
\end{align}
which, by the definition of $B^\lm$ in \eqref{defb}, implies that
\begin{align*}
\begin{split}
 \fint_0^\lm\!\!\!\int_{\mathbb{T}^{d}} |\mathcal{D} B^\lm_{\bar{i}j}(x,t,y,s) |^2 dyds \leq C \|\mathcal{D} A^\lm\|^2_{\infty}.
\end{split}
\end{align*} Note that $$ \fint_0^\lm\!\!\!\int_{\mathbb{T}^{d}} |B^\lm_{\bar{i}j}(x,t,y,s) |^2 dyds \leq C.$$
By using Parseval's theorem, it is not difficult to derive \eqref{eshuab} in the same manner as \eqref{huab2} and \eqref{huab3}.
\end{proof}

\subsection{An $\va$-smoothing operator}
 Let $ \varphi_1(t)\in C_c^\infty((-1 , 1 )), \,\varphi_2(x)\in C_c^\infty(B(0,1)) $ be fixed nonnegative functions such that $$\int_{\R} \varphi_1(t) dt=1  \quad\text{and}\quad  \int_{\mathbb{R}^{d}} \varphi_2(x)dx =1.$$
Set   $\varphi_{1,\varepsilon} (t)=\frac{1}{\varepsilon^{2}} \varphi_1(t/\va^{2}), \,\varphi_{2,\varepsilon}(x)=\frac{1}{\varepsilon^{d}} \varphi_2(x/\va).$   For functions of form  $f^\va(x,t)=f(x,t,x/\va,t/\va^2)$, we define
\begin{align}\label{sva}
\begin{split}
&\widetilde{S}_\varepsilon(f^\va)(x,t)=\int_{\mathbb{R}^{d}}   f(z,t, x/\va, t/\va^2)\varphi_{2,\va} (x-z)dz,\\
&S_\varepsilon(f^\va)(x,t)=\int_{\mathbb{R}^{d+1}} f( z,  \tau, x/\va, t/\va^2)\varphi_{1,\va} (t-\tau) \varphi_{2,\va} (x-z) dz d\tau.
\end{split}
\end{align}
By the definition, it is obvious that
\begin{align*}
S_\va(f^\va)(x,t)=\int_{\R} \wt{S}_\va(f^\va)(x,\tau ,x/\va,t/\va^2)\varphi_{1,\va}(t-\tau)  d\tau.
\end{align*}

\begin{lemma}\label{le2.0}
Let $g=g(x,t,y,s) \in L^\infty(\R^{d+1}; L^p_{loc}(\R^{d+1}))$ and $f \in L^p(\R^{d+1})$. Then for $\ell=0$ or $1$, and any $1\leq p<\infty$, we have for $\de,\va>0$,
\begin{align}
&\|S_\de(g^\va   \na ^\ell f)\|_{L^p(\R^{d+1})}\nonumber\\
 & \leq C  \|  f\|_{L^p(\R^{d+1})} \Big\{\de^{-\ell}\sup_{ (x,t)\in \R^{d+1}} \sup_{ (y',s')\in \R^{d+1}}\Big(\fint_{\mathcal{Q}_{\de/\va}(y',s')}|g(x,t,y,s)  |^pdyds\Big)^{1/p}\label{le20-re0}\\
 & \quad + \sup_{ (x,t)\in \R^{d+1}}  \sup_{ (y',s')\in \R^{d+1}}\Big(\fint_{\mathcal{Q}_{\de/\va}(y',s')}|\na^{\ell}_x g(x,t,y,s) |^p dyds\Big)^{1/p}\Big\} ,\nonumber\\
&\| \na (S_\de(g^\va    f))\|_{L^p(\R^{d+1})}\nonumber\\
& \leq C \|  f\|_{L^p(\R^{d+1})}  \Big\{\de^{-1} \sup_{ (x,t)\in \R^{d+1}} \sup_{ (y',s')\in \R^{d+1}}   \Big(\fint_{\mathcal{Q}_{\de/\va}(y',s')}|g(x,t,y,s)  |^pdyds\Big)^{1/p}\label{le20-re1}\\
 & \quad +  \va^{-1}\sup_{(x,t)\in \R^{d+1}} \sup_{ (y',s')\in \R^{d+1}}\Big(\fint_{\mathcal{Q}_{\de/\va}(y',s')}\!|\na_y g(x,t,y,s) |^p dyds\Big)^{1/p}\Big\},\nonumber
\end{align}
where $g^\va=g(x,t,x/\va,t/\va^2)$, $  \mathcal{Q}_{r}(x,t)=B(x,r) \times (t-r^2, t+r^2)$, and  $C$ depends only on $d$ and $p.$
\end{lemma}
\begin{proof}  For \eqref{le20-re0}, we provide the proof for the case $\ell=1$, and omit the details for the simpler case $\ell=0$. Note that
\begin{align*}
\big|S_\de(g^\va   \na   f) \big|&=\Big|\int_{\R^{d+1}} g(z,\tau,x/\va,t/\va^2)   \na  f(z,\tau) \varphi_{2,\de}(x-z) \varphi_{1,\de}(t-\tau)  dzd\tau\Big|\\
&\leq \Big|\int_{\R^{d+1}}  \na_z g (z,\tau,x/\va,t/\va^2)   f(z,\tau) \varphi_{2,\de}(x-z) \varphi_{1,\de}(t-\tau)  dzd\tau\Big|\\
&\quad+\Big|\int_{\R^{d+1}} g(z,\tau,x/\va,t/\va^2) f(z,\tau) \na \varphi_{2,\de}(x-z) \varphi_{1,\de}(t-\tau)  dzd\tau\Big|.\end{align*}
By H\"{o}lder's inequality and Fubini's theorem,
\begin{align*}
&\|S_\de(g^\va  \na  f) \|^p_{L^p(\R^{d+1})}\\
&\leq C  \int_{\R^{d+1}}\fint_{B(x,\de)}\fint_{I(t,\de^2)}  |\na_z g(z,\tau,x/\va,t/\va^2)|^p   |f(z,\tau)|^p    dzd\tau dxdt\\
&\quad+ C\de^{-p} \int_{\R^{d+1}} \fint_{B(x,\de)}\fint_{I(t,\de^2)} |g(z,\tau,x/\va,t/\va^2)|^p |f(z,\tau)|^p  dzd\tau dxdt\\
&\leq C \sup_{(z,\tau)\in\R^{d+1}} \sup_{(y',s')\in\R^{d+1}}\fint_{\mathcal{Q}_{\de/\va}(y',s')} |\na_z g(z,\tau,y,s)|^p dyds  \int_{\R^{d+1}} |f(z,\tau)|^p   dzd\tau  \\
&\quad+ C \de^{-p} \sup_{(z,\tau)\in\R^{d+1}}\sup_{(y',s')\in\R^{d+1}} \fint_{\mathcal{Q}_{\de/\va}(y',s')}  |g(z,\tau,y,s)|^p dyds  \int_{\R^{d+1}} |f(z,\tau)|^p   dzd\tau,
\end{align*}
 where $I(t,\de^2)=(t-\de^2,t+\de^2).$ From the above inequality \eqref{le20-re0} follows directly.

It remain to prove \eqref{le20-re1}. Observe that
\begin{align*}
\big| \na  (S_\de(g^\va    f))\big|
&\leq \va^{-1}\Big|\int_{\R^{d+1}} \na_y g(z,\tau,x/\va,t/\va^2)   f(z,\tau) \varphi_{2,\de}(x-z) \varphi_{1,\de}(t-\tau)  dzd\tau\Big|\\
&\quad+\Big|\int_{\R^{d+1}} g(z,\tau,x/\va,t/\va^2) f(z,\tau) \na \varphi_{2,\de}(x-z) \varphi_{1,\de}(t-\tau)  dzd\tau\Big|.\end{align*}
The desired estimate then follows by H\"{o}lder's inequality and Fubini's theorem as above.
\end{proof}

\begin{lemma}\label{le2.2}
Let $S_\varepsilon$ be defined as in \eqref{sva}. Assume that  $g=g(x,t,y,s) \in L^\infty(\R^{d+1}\times \R^{d+1})$ and   $\na_x g(x,t,y,s), \pa_t g(x,t,y,s)\in L^\infty(\R^{d+1}\times\R^{d+1}).$
Then for any $f\in L^p(\R; W^{2,p}(\R^{d}))\cap W^{1,p}(\R; L^{p}(\R^{d})) $,
\begin{align}\label{le2.2re1}
\begin{split}
 &\|g^\varepsilon \na  f- S_\de(g^\va \na  f)\|_{L^p(\mathbb{R}^{d+1})}\\
&\leq C \|\na f\|_{L^p(\R^{d+1})}  \big\{\de^2  \|\pa_tg \|_{\infty}+ \de\|\na_x g \|_{\infty}\big\}\\
&\quad  + C  \de \big\{\de\|\na_x   g\|_{\infty}+   \| g\|_{\infty} \big\} \big\{\|\na^2  f\|_{L^p(\R^{d+1})} +\|\pa_t f\|_{L^p(\R^{d+1})} \big\},
\end{split}
\end{align}
where $C$ depends only on $d$ and $p$.
\end{lemma}
\begin{proof}
Observe that
\begin{align}\label{plle2.2-1}
\begin{split}
& \big|g^\varepsilon(x,t)  \na  f(x,t)- S_\de(g^\va \na  f)(x,t)\big| \\
&\leq \big|g^\varepsilon(x,t)  \na f(x,t)-\widetilde{S}_\de (g^\va \na f)(x,t)\big|+ \big|\widetilde{S}_\de(g^\va \na f)(x,t)-S_\de (g^\va \na f)(x,t)\big|\\
&\leq C \fint_{B(x,\de)} \big| g(x,t,x/\va,t/\va^2) \na f(x,t) -g(z,t,x/\va,t/\va^2)  \na  f(z,t) \big|   dz \\
 &\quad +C \int_{\R} \varphi_{1} (\tau) \big| \widetilde{S}_\de (g^\va \na  f)(x,t,x/\va,t/\va^2)- \widetilde{S}_\de (g^\va \na f) (x, t-\de^2\tau,x/\va,t/\va^2)\big| d\tau\\
 &\doteq I_1+I_2.
\end{split}
\end{align}
By the inequality,
\begin{align*}
 \fint_{B(x,\de)} |u(x)-u(z)|dz \leq C \int_{B(x,\de)} \frac{|\na u(z)|}{|z-x|^{d-1}} dz,
\end{align*}
where $C$ depends only on $d$, we deduce that
\begin{align*}
I_1\leq  C \|\na_x g \|_{\infty} \int_{B(x,\de)} \frac{|\na f(z,t)|}{|z-x|^{d-1}} dz+ C\| g \|_{\infty} \int_{B(x,\de)} \frac{|\na^2   f(z,t)|}{|z-x|^{d-1}} dz.
\end{align*}
Note that the operator $$ \mathcal{T} h=\int_{B(x,\de)} \frac{  h(z) }{|z-x|^{d-1}} dz $$
is bounded in $L^p(\R^d)$ with
$   \|\mathcal{T} h  \|_{L^p(\R^d)}\leq C \de \| h\|_{L^p(\R^d)}
$ for any $1\leq p\leq \infty$.
Therefore,
\begin{align}\label{plle2.2-3}
\|I_1\|_{L^p(\R^{d+1})}  \leq C \de \big\{ \|\na_x g \|_{\infty} \|\na  f \|_{L^p(\R^{d+1})}  +  \| g \|_{\infty}  \|\na^2  f \|_{L^p(\R^{d+1})} \big\}.
\end{align}

To deal with $ I_2$, we note that
\begin{align*}
&\big| \widetilde{S}_\de (g^\va \na  f)(x,t,x/\va,t/\va^2)- \widetilde{S}_\de (g^\va \na  f) (x, t-\de^2\tau,x/\va,t/\va^2)\big|\\
&\leq \Big|\int_{\R^d} \big[g(z,t,x/\va,t/\va^2) - g(z,t-\de^2\tau,x/\va,t/\va^2)\big] \na f(z,t) \varphi_{2,\de}(x-z)dz\Big|\\
&\quad+\Big|\int_{\R^d} g(z,t-\de^2\tau,x/\va,t/\va^2) \big[ \na f(z,t)- \na f(z,t-\de^2\tau) \big] \varphi_{2,\de}(x-z)dz\Big|\\
 &\leq \Big|\int_{\R^d}\int_0^{\de^2\tau} (\pa_t g)(z,t-s,x/\va,t/\va^2)\na f(z,t) \varphi_{2,\de}(x-z)dsdz\Big|\\
&\quad+\Big|\int_{\R^d} \int_0^{\de^2\tau}g(z,t-\de^2\tau,x/\va,t/\va^2)  \na \pa_t f(z,t-s) \varphi_{2,\de}(x-z) dsdz\Big|.
\end{align*}
By integration by parts and Fubini's theorem, we deduce that
\begin{align*}
\begin{split}
&\big\| \widetilde{S}_\de (g^\va \na  f)(x,t,x/\va,t/\va^2)- \widetilde{S}_\de (g^\va \na  f) (x, t-\de^2\tau,x/\va,t/\va^2)\big\|^p_{L^p(\R^{d+1})}\\
&\leq   \int_{\R^{d+1}}\Big|\int_{\R^d} \int_0^{\de^2\tau}(\pa_t g)(z,t-s,x/\va,t/\va^2)  \na f(z,t) \varphi_{2,\de}(x-z)dzds\Big|^p dxdt  \\
&\quad+  \int_{\R^{d+1}}\Big|\int_{\R^d}  \int_0^{\de^2\tau}(\na_xg)(z,t-\de^2\tau,x/\va,t/\va^2) \pa_t f(z,t-s) \varphi_{2,\de}(x-z)  dzds\Big|^p dxdt \\
&\quad+  \int_{\R^{d+1}}\Big|\int_{\R^d}  \int_0^{\de^2\tau} g(z,t-\de^2\tau,x/\va,t/\va^2)  \pa_t f(z,t-s)  \na\varphi_{2,\de}(x-z) dz ds\Big|^p dxdt \\
&\leq C \big\{\big(\de^2\tau \|\pa_t g\|_\infty  \|\na f\|_{L^p(\R^{d+1})}\big)^p+ \big(\de^2 \tau \|\na_x g\|_{\infty}  \|\pa_t f\|_{L^p(\R^{d+1})}\big)^p+ \big(\de \tau \|g\|_{\infty} \|\pa_t f\|_{L^p(\R^{d+1})}\big)^p\big\},
\end{split}
\end{align*}
which, combined with Minkowski's inequality, implies that
\begin{align}\label{plle2.2-4}
\begin{split}
\|I_2\|_{L^p(\R^{d+1})}
 &\leq  \int_{\R} \varphi_1(\tau)  \big\| \widetilde{S}_\de (g^\va \na  f)(x,t,x/\va,t/\va^2)\\
 &\quad\quad\quad\quad\quad\quad-\widetilde{S}_\de (g^\va \na f) (x, t\!-\!\de^2\tau,x/\va,t/\va^2)\big\|_{L^p(\R^{d+1})}  d\tau\\
 &\leq C \big(\de^2 \|\na_x g\|_{\infty} +  \de \|g\|_{\infty}\big)  \|\pa_t f\|_{L^p(\R^{d+1})}  +C \de^2  \|\pa_tg \|_{\infty} \|\na  f\|_{L^p(\R^{d+1})}.\end{split}\end{align}
By taking  \eqref{plle2.2-3} and \eqref{plle2.2-4} into \eqref{plle2.2-1}, one derives \eqref{le2.2re1} immediately.
 \end{proof}

\section{Approximation of  $\pa_t+\mathcal{L}^\lm_\va$}
In this section, we investigate the approximation of solutions to $\pa_t u^\lm_\va+\mathcal{L}^\lm_\va u^\lm_\va=F$ by solutions to $\pa_t u^\lm_0 +\mathcal{L}_0^\lm u^\lm_0=F$, where  $\mc{L}^\lm_\varepsilon, \mathcal{L}^\lm_0$ are given by \eqref{Llm} and  \eqref{nonl0}.

Let $u^\lm_\va$ be a weak solution to
 \begin{align} \label{ap-eq-va}
\pa_t u^\lm_\va +\mathcal{L}^\lm_\va u^\lm_\va  = F   \quad  \text{in } Q_2,
\end{align}
and  $u^\lm_0$ the weak solution to
  \begin{align}\label{ap-eq-ho}
\pa_t u^\lm_0+\mathcal{L}^\lm_0 u^\lm_0 =F \quad \text{ in }   Q_1,
\quad \text{and} \quad u^\lm_0=u^\lm_\va   \text{ on } \pa_p Q_1.
\end{align} Here and after $Q_r=Q_r(0,0)$. Let $\chi^\lm$ and $ \mathfrak{B}^\lm$ be the matrices of correctors and flux correctors given by \eqref{chi} and Lemma \ref{le-fcon}. Let $\de=(1+\sqrt{\lm})\va$ and assume that $\de<1/20.$
Define \begin{align} \label{def-w}
\begin{split}
w_\varepsilon^\lm &=u_\varepsilon^\lm-u_0^\lm -\varepsilon S_\de ( (\chi^\lm)^\varepsilon\na  u_0^\lm)\eta_\de   +\varepsilon^{2}  S_\de ( (\pa_{x_i }\mathfrak{B}^\lm_{i (d+1) j})^\varepsilon\pa_{x_j} u_0^\lm)\eta_\de \\
&\quad+ \varepsilon^{2}  S_\de ( (\mathfrak{B}^\lm_{i (d+1) j})^\varepsilon \pa_{x_i } \pa_{x_j} u_0^\lm)\eta_\de  + \varepsilon^{2}  S_\de ( (\mathfrak{B}^\lm_{i (d+1) j})^\varepsilon\pa_{x_j} u_0^\lm) \pa_{x_i } \eta_\de ,
 \end{split}
\end{align}
where $f^\va(x, t)=f(x,t,x/\va, t/\va^2)$, and $ \eta_\de   \in C_c^\infty(\R^{d+1})$ is a cut-off function such that $0\leq \eta_\de  \leq 1$ and
\begin{align}\label{rho}
 \begin{split}
 &\eta_\de  =1  \,\,\text{ in } Q_{1-4\de},\quad  \eta_\de  =0 \,\,\text{ in } Q_1\!\setminus\!Q_{1-3\de}, \\
& |\na \eta_\de |\leq C \de^{-1}, \quad |\pa_t \eta_\de |+ |\nabla^2 \eta_\de  |\leq C \de^{-2}.
\end{split}\end{align}.

\begin{lemma}\label{leap-2}
Assume that $A$ satisfies conditions (\ref{cod1}), (\ref{cod2})  and
\begin{align}\label{ainfty}
\|\na_x A\|_{\infty} +\|\pa_t A\|_{\infty}< \infty.
\end{align}
Let $w^\lm_\va$  be defined as in \eqref{def-w}.
  Then we have
\begin{align} \label{leap-2re}
\begin{split}
\|\na w^\lm_\va\|_{L^2(Q_1)}
&\leq  C \big\{\de\|\na_x A^\lm \|_{\infty} +1 \big\}  \|\na u^\lm_0\|_{L^2(Q_1\setminus Q_{1-5\de})} \\
&\quad
 +C \de\big(\|\na_x A^\lm\|_{\infty}\!+\!\de \|\pa_t A^\lm\|_{\infty}\big)\|\na u_0^\lm\|_{L^2(Q_1)}  \\
&\quad+ C\de \big (\de\|\na_x A^\lm\|_{\infty} +1 \big)
  \big (\|\na^2 u_0^\lm\|_{L^2(Q_{1-2\de})} +\|\pa_t u_0^\lm\|_{L^2(Q_{1-2\de})}\big),
  \end{split}
\end{align}
where $C$ depends only on $d$ and $\mu$.
\end{lemma}

\begin{proof} By the definition of $A^\lm$, it is obvious that $\|\na_xA\|_\infty=\|\na_xA^\lm\|_\infty$ and $\|\pa_tA\|_\infty=\|\pa_tA^\lm\|_\infty.$
By \eqref{defb} and direct calculations,
\begin{align}\label{pl311}
(\pa_t +\mathcal{L}^\lm_\va)  w^\lm_\va &=-\text{div}\big\{ (\widehat{A^\lm}
 -  {(A^\lm)^\va}) \na u^\lm_0- S_\de ( (\widehat{A^\lm}
 -  {(A^\lm)^\va}) \na u^\lm_0)\eta_\de  \big\}  +\text{div}  \{ S_\de ((B^\lm)^\va \na u^\lm_0)\eta_\de \}  \nonumber\\
&\quad - \text{div}\big\{    S_\de ({(A^\lm)^\va}(\na_y\chi^\lm)^\va \na u^\lm_0)\eta_\de  \big\}+\text{div}\big\{    {(A^\lm)^\va}  S_\de ( (\na_y\chi^\lm)^\va \na u^\lm_0)\eta_\de  \big\} \nonumber\\
&\quad+\va \text{div} \big\{{(A^\lm)^\va}  S_\de ( (\na_x\chi^\lm)^\va \na u^\lm_0)\eta_\de  +{(A^\lm)^\va}  S_\de ( (\chi^\lm)^\va \na^2 u^\lm_0)\eta_\de \big\}\nonumber\\
&\quad+\va \text{div} \big\{ {(A^\lm)^\va}  S_\de ( (\chi^\lm)^\va \na u^\lm_0) \na \eta_\de  \big\}- \va \pa_t  \big\{S_\de ( (\chi^\lm)^\va  \na u^\lm_0)\eta_\de  \big\} \\
&\quad+(\pa_t +\mathcal{L}^\lm_\va)  \big\{ \va^2  S_\de ( ( \pa_{x_i}\mathfrak{B}^\lm_{i (d+1) j})^\va \pa_{x_j}u^\lm_0) \eta_\de   \big\}\nonumber\\
&\quad  +(\pa_t +\mathcal{L}^\lm_\va)  \big\{ \va^2  S_\de ( (\mathfrak{B}^\lm_{i (d+1) j})^\va \pa_{x_i} \pa_{x_j}u^\lm_0)\eta_\de   \big\}\nonumber\\
&\quad  +(\pa_t +\mathcal{L}^\lm_\va)  \big\{ \va^2  S_\de ( (\mathfrak{B}^\lm_{i (d+1) j})^\va \pa_{x_j}u^\lm_0)\pa_{x_i}\eta_\de   \big\},\nonumber
\end{align}
where $(B^\lm)^\va(x,t)=B^\lm(x,t,x/\va,t/\va^2)$ is defined in \eqref{defb}.
By \eqref{defb} and \eqref{huab1}, we have
  \begin{align}\label{pl313}
  \begin{split}
 \pa_t [S_\de ( (\chi^\lm)^\va   \na u^\lm_0)\eta_\de ]&=- \pa_t [S_\de ( (\pa_{y_k} \mathfrak{B}^\lm_{k (d+1)j} )^\va  \pa_{x_j} u^\lm_0)\eta_\de  ]  \\
  &=-\va \pa_t \pa_{x_k} \big\{ S_\de (  (\mathfrak{B}^\lm_{k(d+1)j})^\va  \pa_{x_j} u^\lm_0) \eta_\de \big\}\\
  &\quad+ \va \pa_t \big\{ S_\de ( (\pa_{x_k} \mathfrak{B}^\lm_{k(d+1)j})^\va  \pa_{x_j} u^\lm_0) \eta_\de  \big\}\\
  &\quad+ \va \pa_t \big\{ S_\de (  (\mathfrak{B}^\lm_{k(d+1)j})^\va \pa_{x_k}  \pa_{x_j} u^\lm_0)\eta_\de \big\}\\
  &\quad
   + \va \pa_t \big\{ S_\de (  (\mathfrak{B}^\lm_{k(d+1)j})^\va  \pa_{x_j} u^\lm_0) \pa_{x_k}\eta_\de \big\},
   \end{split}
  \end{align}
and
  \begin{align} \label{pl314}
\text{div} \big\{ S_\de (    (B^\lm)^\va   \na u^\lm_0)\eta_\de  \big\}
 &=  \pa_{x_i} \big\{ S_\de (  (\pa_{y_k}  \mathfrak{B}^\lm_{kij})^\va   \pa_{x_j} u^\lm_0)\eta_\de  \big\}+ \pa_{x_i} \big\{ S_\de (  (\pa_s  \mathfrak{B}^\lm_{(d+1)ij})^\va   \pa_{x_j} u^\lm_0)\eta_\de \big\} \nonumber\\
  &= \va \pa_{x_i} \big\{ \pa_{x_k} [S_\de ((\mathfrak{B}^\lm_{kij})^\va   \pa_{x_j} u^\lm_0)\eta_\de  ]\big\}-\va \pa_{x_i} \big\{  S_\de ((\pa_{x_k}\mathfrak{B}^\lm_{kij})^\va   \pa_{x_j} u^\lm_0)\eta_\de  \big\}\nonumber\\
  &\quad-\va \pa_{x_i} \big\{  S_\de ((\mathfrak{B}^\lm_{kij})^\va   \pa_{x_k}  \pa_{x_j} u^\lm_0)\eta_\de   \big\} -\va \pa_{x_i} \big\{  S_\de ((\mathfrak{B}^\lm_{kij})^\va   \pa_{x_j} u^\lm_0) \pa_{x_k} \eta_\de   \big\}\nonumber\\
  &\quad+\va^2 \pa_{x_i}\pa_t \big\{ S_\de (  (\mathfrak{B}^\lm_{(d+1)ij})^\va   \pa_{x_j} u^\lm_0)\eta_\de  \big\}\\
  &\quad
  - \va^2\pa_{x_i} \big\{ S_\de (  (\pa_t\mathfrak{B}^\lm_{(d+1)ij})^\va  \pa_{x_j} u^\lm_0) \eta_\de  \big\}\nonumber\\
  &\quad- \va^2\pa_{x_i}  \big\{ S_\de (  (\mathfrak{B}^\lm_{(d+1)ij})^\va   \pa_t \pa_{x_j} u^\lm_0) \eta_\de  \big\}\nonumber\\
  &\quad- \va^2\pa_{x_i}  \big\{ S_\de (  (\mathfrak{B}^\lm_{(d+1)ij})^\va  \pa_{x_j} u^\lm_0) \pa_t\eta_\de   \big\}.\nonumber
  \end{align}
 By taking \eqref{pl313} and \eqref{pl314} into \eqref{pl311}, and using the skew-symmetry of $\mathfrak{B}^\lm$, we derive that
 \begin{equation}
\label{leap-1re}
\aligned
 ( \pa_t+ \mathcal{L}^\lm_\va) w^\lm_\va
  &=-\text{div}\big\{ (\widehat{A^\lm}
 -  {(A^\lm)^\va}) \na u^\lm_0- S_\de ( (\widehat{A^\lm} -  {(A^\lm)^\va}) \na u^\lm_0)\eta_\de  \big\}   \\
 &\quad - \text{div}\big\{    S_\de ({(A^\lm)^\va}(\na_y\chi^\lm)^\va \na u^\lm_0)\eta_\de  -   {(A^\lm)^\va}  S_\de ( (\na_y\chi^\lm)^\va \na u^\lm_0)\eta_\de  \big\} \\
 &\quad+\va \text{div} \big\{{(A^\lm)^\va}  S_\de ( (\na_x\chi^\lm)^\va \na u^\lm_0)\eta_\de \big\}  + \va\text{div}\big\{{(A^\lm)^\va}  S_\de ( (\chi^\lm)^\va \na^2 u^\lm_0) \eta_\de  \big\} \\
 &\quad+  \va\text{div}\big\{{(A^\lm)^\va}  S_\de ( (\chi^\lm)^\va \na u^\lm_0)\na\eta_\de    \big\}-\va \pa_{x_i} \big\{  S_\de ((\pa_{x_k}\mathfrak{B}^\lm_{kij})^\va   \pa_{x_j} u^\lm_0)\eta_\de \big\}\\
 &\quad -\va \pa_{x_i} \big\{  S_\de ((\mathfrak{B}^\lm_{kij})^\va    \pa_{x_j} u^\lm_0)\pa_{x_k}\eta_\de  \big\} -\va \pa_{x_i} \big\{  S_\de ((\mathfrak{B}^\lm_{kij})^\va   \pa_{x_k}  \pa_{x_j} u^\lm_0)\eta_\de \big\}\\
  &\quad- \va^2\pa_{x_i} \big\{ S_\de (  (\pa_t\mathfrak{B}^\lm_{(d+1)ij})^\va   \pa_{x_j} u^\lm_0)\eta_\de  \big\}\\
  &\quad
  - \va^2\pa_{x_i}  \big\{ S_\de (  (\mathfrak{B}^\lm_{(d+1)ij})^\va   \pa_t \pa_{x_j} u^\lm_0) \eta_\de  \big\}\\
  &\quad - \va^2\pa_{x_i}  \big\{ S_\de (  (\mathfrak{B}^\lm_{(d+1)ij})^\va     \pa_{x_j} u^\lm_0)\pa_t\eta_\de \big\} \\
  &\quad- \va^2 \text{div} \big\{ (A^\lm)^\va \na \big[ S_\de ( (\mathfrak{B}^\lm_{i (d+1) j})^\va \pa_{x_i} \pa_{x_j}u^\lm_0)\eta_\de   \big]\big\}\\
  &\quad - \va^2 \text{div} \big\{ (A^\lm)^\va \na \big[ S_\de ( (\mathfrak{B}^\lm_{i (d+1) j})^\va   \pa_{x_j}u^\lm_0)\pa_{x_i}\eta_\de   \big]\big\} \\
 &\quad -\va^2 \text{div} \big\{ (A^\lm)^\va \na \big[ S_\de ( ( \pa_{x_i}\mathfrak{B}^\lm_{i (d+1) j})^\va \pa_{x_j}u^\lm_0) \eta_\de  \big]\big\}.
\endaligned
\end{equation}
As a result,
\begin{align}
 &\Big| \int_{-1}^0 \big\langle (\pa_t +\mathcal{L}^\lm_\va)  w^\lm_\va, \phi\rangle dt   \Big|\nonumber\\
   &\leq\int_{Q_1} \Big|\big\{ (\widehat{A^\lm}
 -  {(A^\lm)^\va}) \na u^\lm_0- S_\de ( (\widehat{A^\lm}
 -  {(A^\lm)^\va}) \na u^\lm_0) \eta_\de \big\} \na \phi  \Big| \nonumber\\
  &\quad+  \Big|\int_{Q_1} \big\{ S_\de ({(A^\lm)^\va}(\na_y\chi^\lm)^\va \na u^\lm_0)\eta_\de  -   {(A^\lm)^\va}  S_\de ( (\na_y\chi^\lm)^\va \na u^\lm_0)\eta_\de  \big\}  \na \phi \Big|\nonumber\\
  &\quad+C\va\int_{Q_1}\Big\{  \big| S_\de ( (\na_x\chi^\lm)^\va \na u^\lm_0)\eta_\de  \big| + \big|  S_\de ( (\chi^\lm)^\va  \na^2 u^\lm_0) \eta_\de \big| \Big\} |\na \phi |\nonumber\\
   &\quad + C \va\int_{Q_1}  \big|  S_\de ( (\chi^\lm)^\va  \na u^\lm_0)\na\eta_\de \big|   |\na \phi |\nonumber\\
   &\quad+\va  \int_{Q_1}  \big| S_\de ((\pa_{x_k}\mathfrak{B}^\lm_{kij})^\va  \pa_{x_j} u^\lm_0) \eta_\de  \pa_{x_i} \phi \big|\nonumber\\
    &\quad+\va   \int_{Q_1}  \Big\{ \big|S_\de ((\mathfrak{B}^\lm_{kij})^\va   \pa_{x_k} \pa_{x_j} u^\lm_0)\eta_\de \big| +\big|S_\de ((\mathfrak{B}^\lm_{kij})^\va     \pa_{x_j} u^\lm_0)\pa_{x_k} \eta_\de \big| \Big\}  |\pa_{x_i}  \phi  | \nonumber  \\
 &\quad+  \va^2  \Big|\int_{Q_1}   S_\de (  (\pa_t\mathfrak{B}^\lm_{(d+1)ij})^\va  \pa_{x_j} u^\lm_0) \eta_\de   \pa_{x_i} \phi\Big|\nonumber\\
 &\quad + \va^2  \Big|\int_{Q_1}   S_\de (  (\mathfrak{B}^\lm_{(d+1)ij})^\va    \pa_{x_j} u^\lm_0)\pa_t\eta_\de    \pa_{x_i} \phi\Big|
 \nonumber \\
  &\quad + \va^2  \Big|\int_{Q_1}   S_\de (  (\mathfrak{B}^\lm_{(d+1)ij})^\va   \pa_t \pa_{x_j} u^\lm_0)\eta_\de    \pa_{x_i} \phi\Big|
 \nonumber \\
&\quad+  C\va^2  \int_{Q_1}  \big|\na \big\{ S_\de ( (\mathfrak{B}^\lm_{i (d+1) j})^\va \pa_{x_i} \pa_{x_j}u^\lm_0)\eta_\de    \big\} \big| |\na \phi |\nonumber\\
&\quad+  C\va^2  \int_{Q_1}  \big|\na \big\{ S_\de ( (\mathfrak{B}^\lm_{i (d+1) j})^\va  \pa_{x_j}u^\lm_0)\pa_{x_i}\eta_\de  \big\}\big|  | \na \phi |\nonumber\\
&\quad+  C \va^2 \int_{Q_1}   \big|\na  \big\{ S_\de ( (\pa_{x_i}\mathfrak{B}^\lm_{i (d+1) j})^\va \pa_{x_j}u^\lm_0)\eta_\de  \big\}\big|  | \na \phi|\nonumber\\
&\doteq I_1+\cdot\cdot\cdot+I_{12}.\label{pl351}
\end{align}

Next let us estimate  $I_1$--$I_{12}$ one by one. Note that
\begin{align*}
&\big|  (\widehat{A^\lm}
 -  {(A^\lm)^\va}) \na u^\lm_0- S_\de ( (\widehat{A^\lm}
 -  {(A^\lm)^\va}) \na u^\lm_0) \eta_\de   \big|\\
 &= \big| (\widehat{A^\lm}
 -  {(A^\lm)^\va}) \na u^\lm_0 (1-\eta_\de ) \big|
 + | S_\de ( (\widehat{A^\lm}
 -  {(A^\lm)^\va}) \na u^\lm_0)\eta_\de  -(\widehat{A^\lm}-{(A^\lm)^\va}) \na u^\lm_0\eta_\de \big|.
 \end{align*}
By the Cauchy inequality and Lemma \ref{le2.2},
 \begin{align*}
I_1& \leq  C \|\na u^\lm_0\|_{L^2(Q_1\setminus Q_{1-5\de})} \|\na \phi\|_{L^2(Q_1\setminus Q_{1-5\de})} \\
 &\quad+ C\de \big ( \de  \|\pa_t A^\lm   \|_{\infty} +\|\na_x A^\lm \|_{\infty}\big) \|\na u^\lm_0\|_{L^2( Q_{1-2\de})} \|\na \phi\|_{L^2(Q_1)}\\
 &\quad+C\de\big( \de \|\na_x A^\lm \|_{\infty}+1\big) \big\{ \|\na^2 u^\lm_0\|_{L^2( Q_{1-2\de})} +\|\pa_t u^\lm_0\|_{L^2( Q_{1-2\de})}  \big\} \|\na \phi\|_{L^2(Q_1)},
\end{align*}
where we have also used the observation (see \eqref{lecore2})
\begin{align*}
\|\pa_t \widehat{A^\lm} \|_{\infty}\leq C \|\pa_t A^\lm   \|_{\infty}  \quad\text{and}\quad \|\na  \widehat{A^\lm} \|_{\infty}\leq C \|\na_x A^\lm  \|_{\infty}.
\end{align*}

To estimate $I_2$, we observe that    \begin{align*}
  &| S_\de ({(A^\lm)^\va}(\na_y\chi^\lm)^\va \na u^\lm_0) -   {(A^\lm)^\va}  S_\de ( (\na_y\chi^\lm)^\va \na u^\lm_0)|\\
  &\leq C \Big|\fint_{B(x,\de)}\!\fint_{I(t,\de^2)}\!\!\!\big[A^\lm(z,\tau,x/\va,t/\va^2)\!-\!A^\lm(x,t,x/\va,t/\va^2)\big] \na_y\chi^\lm(z,\tau,x/\va,t/\va^2)  \na u^\lm_0 (z,\tau)dzd\tau\Big|\\
  &\leq C \de\big\{\|\na_x A^\lm \|_{\infty}+\de\|\pa_t A^\lm  \|_{\infty}\big\}\fint_{B(x,\de)}\!\fint_{I(t,\de^2)} |  \na_y\chi^\lm(z,\tau,x/\va,t/\va^2)||\na u^\lm_0 (z,\tau)|dzd\tau,
\end{align*}
where $I(t,\de^2)=(t-\de^2, t+\de^2)$.
By Cauchy's inequality and Fubini's theorem,
\begin{align*}
I_2&\leq  C \de\big\{\|\na_x A^\lm \|_{\infty}+\de \|\pa_t A^\lm  \|_{\infty}\big\} \|\na \phi\|_{L^2(Q_1)} \\
 &\quad\times \big\|\eta_\de  \fint_{B(x,\de)}\!\fint_{I(t,\de^2)} |\na_y\chi^\lm(z,\tau,x/\va,t/\va^2)||\na u^\lm_0 (z,\tau)|dzd\tau \big\|_{L^2(Q_1)}\\
 &\leq C \de\big\{\|\na_x A^\lm \|_{\infty}+\de \|\pa_t A^\lm  \|_{\infty}\big\} \|\na u^\lm_0\|_{L^2(Q_1)}\|\na \phi\|_{L^2(Q_1)}.
\end{align*}
For the last step, we have used the estimate
\begin{align}\label{o1o1}
\fint_{Q_{1+\sqrt{\lm}}} | \na_y \chi^\lm (x,t,y,s)|^2 + | \chi^\lm (x,t,y,s)|^2 dyds\leq C,
\end{align}
which is direct consequence of  \eqref{lecore0}.

We pass to $I_3$--$I_6$.  Note that \eqref{lecore3} implies
 \begin{align}\label{o1o2}
\fint_{Q_{1+\sqrt{\lm}}} | \mathcal{D} \na_y \chi^\lm (x,t,y,s)|^2 + |\mathcal{D }\chi^\lm (x,t,y,s)|^2 dyds\leq C \|\mathcal{D }A^\lm\|_{\infty}
\end{align}
 for $\mathcal{D}=\na_x$ or $\pa_t$.
By using \eqref{o1o1}, \eqref{o1o2} and similar estimates for $\mathfrak{B}^\lm$ (which can be deduced from \eqref{huab2}-\eqref{eshuab}), and Lemma \ref{le2.0}, we deduce that
\begin{align*}
I_3+...+I_6 &\leq C \va \big\{\|\na_x A^\lm \|_{\infty} \| \na u^\lm_0\|_{L^2(Q_1)} + \|\na^2 u^\lm_0\|_{L^2( Q_{1-2\de})}\big\}   \|\na \phi\|_{L^2(Q_1)}\\
& \quad+ C  \|\na u^\lm_0\|_{L^2(Q_1\setminus Q_{1-5\de})} \|\na \phi\|_{L^2(Q_1\setminus Q_{1-5\de})}.
\end{align*}
In a similar way, we can bound $I_7$--$I_9$ as following
 \begin{align*}
I_7+I_8+I_9&\leq C  \de^2\|\pa_t A^\lm  \|_{\infty} \|\na u^\lm_0\|_{L^2(Q_1)}    \|\na \phi\|_{L^2(Q_1)}\\
& \quad+  C  \|\na u^\lm_0\|_{L^2(Q_1\setminus Q_{1-5\de})} \|\na \phi\|_{L^2(Q_1\setminus Q_{1-5\de})}\\
&\quad + \de ( 1+ \de\|\na_x A^\lm \|_{\infty} ) \|\pa_t u^\lm_0\|_{L^2(Q_{1-2\de})}  \|\na \phi\|_{L^2(Q_1)},
 \end{align*}where we have used the observation $(1+\lm)\va^2\leq \de^2.$

For $I_{10}$, we use Lemma \ref{le2.0} and the estimates similar to \eqref{o1o1} and \eqref{o1o2} for  $\mathfrak{B}^\lm$, to deduce that
\begin{align*}
 I_{10}&\leq C\va^2 \| S_\de ( (\mathfrak{B}^\lm_{i (d+1) j})^\va \pa_{x_i} \pa_{x_j}u^\lm_0)\na \eta_\de  \|_{L^2(Q_1)}\|\na \phi\|_{L^2(Q_1\setminus Q_{1-5\de})}\\
   & \quad+C\va^2\| \na_x \big\{ S_\de ( ( \mathfrak{B}^\lm_{i (d+1) j})^\va \pa_{x_i} \pa_{x_j}u^\lm_0)\big\} \eta_\de \|_{L^2(Q_1)}\|\na \phi\|_{L^2(Q_1)}\\
   &\leq C  (1+   \de\|\na_x A^\lm \|_{\infty} )  \|\na u^\lm_0\|_{L^2(Q_1\setminus Q_{1-5\de})}  \|\na \phi\|_{L^2(Q_1\setminus Q_{1-5\de})}\\
&\quad+ C  \de \|\na^2 u^\lm_0\|_{L^2( Q_{1-2\de})}  \|\na \phi\|_{L^2(Q_1)} .
\end{align*}
Likewise, we can bound $I_{11}$ and $I_{12}$ as follows,
\begin{align*}
I_{11}&\leq C (1+\de \|\na_x A^\lm \|_{\infty})  \|\na u^\lm_0\|_{L^2(Q_1\setminus Q_{1-5\de})}   \|\na \phi\|_{L^2(Q_1\setminus Q_{1-5\de})},\\
 I_{12}&\leq C  \de   \|\na_x A^\lm \|_{\infty} \|\na u^\lm_0\|_{L^2(Q_1\setminus Q_{1-5\de})}  \|\na \phi\|_{L^2(Q_1\setminus Q_{1-5\de})}\\
 &\quad+C\de\|\na_x A^\lm \|_{\infty} \|\na u^\lm_0\|_{L^2(Q_{1})}   \|\na \phi\|_{L^2( Q_{1})}    .\nonumber
\end{align*}
Finally, taking the estimates for $I_1$--$I_{12}$ into \eqref{pl351}, it follows that
\begin{align}\label{pl316}
\begin{split}
& \Big| \int_{-1}^0 \big\langle (\pa_t +\mathcal{L}^\lm_\va)  w^\lm_\va, \phi\rangle dt   \Big|\\
  &\leq C\big\{\de\|\na_x A^\lm \|_{\infty} +1 \big\}  \|\na u^\lm_0\|_{L^2(Q_1\setminus Q_{1-5\de})} \|\na \phi\|_{L^2(Q_1\setminus Q_{1-5\de})} \\
 &\quad+ C\de \big\{ \de  \|\pa_t A^\lm   \|_{\infty} +\|\na_x A^\lm \|_{\infty}\big\} \|\na u^\lm_0\|_{L^2( Q_1)} \|\na \phi\|_{L^2(Q_1)}\\
 &\quad+C\de \big\{ \de \|\na_x A^\lm \|_{\infty}+1\big\} \big\{ \|\na^2 u^\lm_0\|_{L^2( Q_{1-2\de})} +\|\pa_t u^\lm_0\|_{L^2( Q_{1-2\de})}  \big\} \|\na \phi\|_{L^2(Q_1)},
 \end{split}
\end{align}
which gives \eqref{leap-2re} by setting $\phi=w^\lm_\va$.
\end{proof}

\begin{lemma}\label{leap-3}
Assume that $A$ satisfies conditions (\ref{cod1}), (\ref{cod2}) and \eqref{ainfty}.
 Let $u^\lm_\va$ be a weak solution to \eqref{ap-eq-va} and $u^\lm_0$ the weak  solution to \eqref{ap-eq-ho}.  Then for $\de=(1+\sqrt{\lm})\va,$
 \begin{align}\label{leap-3re}
 \begin{split}
  \Big(\fint_{Q_1}  | u^\lm_\va-u^\lm_0  |^2 \Big )^{1/2}
 &\leq  C \big(\de^\sigma \!+\! \de \|\na_x A^\lm\|_{\infty}+ \de^2\|\pa_t A^\lm\|_{\infty}\big) \big(\de \|\na_x A^\lm\|_{\infty}\! +\!1 \big)\\
 & \quad\times \Big\{ \Big (\fint_{Q_2}  | u^\lm_\va  |^2  \Big )^{1/2} + \Big (\fint_{Q_2}  |F |^2 \Big )^{1/2} \Big\},
 \end{split}
 \end{align}
 where $0<\sigma<1 $ and $C$ depend  only on $d$ and $\mu$.
\end{lemma}
\begin{proof}
 Note that
\begin{align*}
(\pa_t  +\mathcal{L}^\lm_0 ) (u^\lm_0  -u^\lm_\va )=(\mathcal{L}^\lm_\va-\mathcal{L}^\lm_0) u^\lm_\va   \quad \text{in } Q_1 \quad\,\text{ and }\,\quad u^\lm_0  -u^\lm_\va =0 \quad\text{on } \pa_p Q_1.
 \end{align*}
By the $W^{1,p}$ estimate of parabolic systems (see e.g., \cite{bw2005}), we have for any $q\geq 2,$
 \begin{align}\label{prleapl-2}
 \fint_{Q_1} |\na u^\lm_0|^q \leq C  \fint_{Q_1} |\na u^\lm_\va |^q,
 \end{align}
where $C$ depends only on $d $ and $\mu.$ Moreover, by Meyer's estimates of parabolic systems (see e.g. \cite{Armstrongapde-2018}), there exists some $q>2$ such that
\begin{align*}
\Big(\fint_{Q_1} |\na u^\lm_\va |^q \Big)^{2/q}\leq C \Big\{\fint_{Q_{3/2}} |\na u^\lm_\va |^2   +  \fint_{Q_{3/2}} |F|^2  \Big\}.
\end{align*}
This, combined with \eqref{prleapl-2} and Caccioppoli's inequality, gives
\begin{align}\label{prleapl-3}
\begin{split}
\Big(\fint_{Q_1} |\na u^\lm_0  |^q \Big)^{2/q}&\leq C\Big\{ \fint_{Q_{3/2}} |\na u^\lm_\va |^2   +  \fint_{Q_{3/2}} |F|^2  \Big\}\\
&\leq C\Big\{ \fint_{Q_{2}} | u^\lm_\va |^2  +  \fint_{Q_{2}} |F|^2  \Big\}
\end{split}
\end{align} for some $q>2$, where $C $ depends only on $d$ and $\mu$.

To move on, let $w_\va^\lm$ be defined as in \eqref{def-w} with $\de=(1+\sqrt{\lm})\va$. We also assume that $\de<1/20,$ for otherwise the estimate \eqref{leap-3re} is trivial.  Observe that
\begin{align}\label{prleapl-4}
\begin{split}
 &\big\|\na \big(u^\lm_\va -u^\lm_0  -\va S_\de ((\chi^\lm)^\va\na u^\lm_0  )\eta_\de \big)\big\|_{L^2(Q_1)} \\
&\leq  \|\na w^\lm_\va\|_{L^2(Q_1)}+  \va^2 \big\| \na \big( S_\de ( (\pa_{x_i }\mathfrak{B}^\lm_{i (d+1) j})^\varepsilon\pa_{x_j} u^\lm_0  )\eta_\de \big)\big\|_{L^2(Q_1)} \\
 &\quad+ \va^2 \big\| \na \big(  S_\de ( (\mathfrak{B}^\lm_{i (d+1) j})^\varepsilon\pa_{x_j} u^\lm_0  ) \pa_{x_i } \eta_\de \big)\big\|_{L^2(Q_1)}\\
&\quad+ \va^2\big\| \na \big( S_\de ( (\mathfrak{B}^\lm_{i (d+1) j})^\varepsilon \pa_{x_i } \pa_{x_j} u^\lm_0  )\eta_\de  \big)\big\|_{L^2(Q_1)}.
 \end{split}\end{align}
An inspection of the estimates $I_{10}$--$I_{12}$ shows that
\begin{align}  \label{prleapl-5}
\begin{split}
\va^{2}\| \na \big( S_\de ( (\pa_{x_i }\mathfrak{B}^\lm_{i (d+1) j})^\varepsilon\pa_{x_j} u^\lm_0  )\eta_\de \big)\|_{L^2(Q_1)}
 & \leq C\de \|\na_x A^\lm \|_{\infty} \|\na u^\lm_0  \|_{L^2(Q_1)},\\
  \va^{2} \| \na \big(  S_\de ( (\mathfrak{B}^\lm_{i (d+1) j})^\varepsilon\pa_{x_j} u^\lm_0  ) \pa_{x_i } \eta_\de \big)\|_{L^2(Q_1)}& \leq  C (1+\de \|\na_x A^\lm \|_{\infty})  \|\na u^\lm_0  \|_{L^2(Q_1\setminus Q_{1-5\de})},
 \end{split}
 \end{align}
 and \begin{align}\label{prleapl-5'}
 \begin{split}
 &\va^{2}\| \na \big( S_\de ( (\mathfrak{B}^\lm_{i (d+1) j})^\varepsilon \pa_{x_i } \pa_{x_j} u^\lm_0  )\eta_\de  \big)\|_{L^2(Q_1)}\\
 &\leq C \big(1+ \de\|\na_x A^\lm \|_{\infty}\big)  \|\na u^\lm_0  \|_{L^2(Q_1\setminus Q_{1-5\de})}+C \de  \|\na^2 u^\lm_0  \|_{L^2( Q_{1-2\de})},
 \end{split}
 \end{align}
 which,  together with \eqref{leap-2re}, \eqref{prleapl-4} and \eqref{prleapl-5}, gives
 \begin{align}\label{prleapl-6}
 \begin{split}
 &\big\|\na \big(u^\lm_\va -u^\lm_0  -\va S_\de ((\chi^\lm)^\va\na u^\lm_0  )\eta_\de \big)\big\|_{L^2(Q_1)} \\
&\leq C\big(\de\|\na_x A^\lm \|_{\infty} +1 \big)  \|\na u^\lm_0  \|_{L^2(Q_1\setminus Q_{1-5\de})}\\
 &\quad+C \de\big(\|\na_x A^\lm \|_{\infty}\!+\!\de \|\pa_t A^\lm \|_{\infty}\big)\|\na u^\lm_0  \|_{L^2(Q_1)}  \\
&\quad+ C\de\big (\de\|\na_x A^\lm \|_{\infty} +1 \big)
  \big \{\|\na^2 u^\lm_0  \|_{L^2(Q_{1-2\de})} +\|\pa_t u^\lm_0  \|_{L^2(Q_{1-2\de})}\big\}.
\end{split}
\end{align}

To bound the right hand side of \eqref{prleapl-6}, we first note that \eqref{prleapl-3} and H\"{o}lder's inequality imply that
\begin{align}\label{prleapl-10}
\begin{split}
 &\int_{Q_1\setminus Q_{1-5\de}}  |\na u^\lm_0   |^2   \leq C \de^{1-\frac{2}{q}} \Big\{ \fint_{Q_{2}} | u^\lm_\va |^2  +  \fint_{Q_{2}} |F|^2  \Big\},\\
 &\int_{Q_1}  |\na u^\lm_0   |^2    \leq C\Big\{ \fint_{Q_{2}} | u^\lm_\va |^2  +  \fint_{Q_{2}} |F|^2  \Big\}.
 \end{split}
\end{align}
On the other hand, thanks to the interior estimates of parabolic systems with Lipschitz coefficients, one has
\begin{align}
\begin{split}
&\fint_{Q_r(x,t)} (|\pa_t u^\lm_0  |^2+|\na^2 u^\lm_0  |^2)\nonumber \\
&\leq    C(\|\na_x A^\lm \|^2_{\infty}+4r^{-2})\fint_{Q_{2r}(x,t)}  |\na u^\lm_0  |^2  +C\fint_{Q_{2r}(x,t)}  |F|^2
\end{split}
\end{align} for any ${Q_{2r}(x,t)} \subseteq Q_2.$
By integration, this implies that
\begin{align} \label{prleapl-11}
\begin{split}
\int_{ Q_{1-2\de}} \big(|\na^2 u^\lm_0  |^2 +|\pa_t u^\lm_0  |^2\big)&\leq C  \|\na_x A^\lm \|^2_{\infty} \int_{Q_1}  |\na u^\lm_0   |^2+C \int_{Q_1} |F|^2 \\
&\quad+ C \int_{Q_{1-\de}} \frac{|\na u^\lm_0  (y,s)|^2dyds}{|dist_p((y,s),\pa_p Q_1)|^2} ,
\end{split}
\end{align}
where $dist_p((y,s),\pa_p Q_1)$ denotes the parabolic distance from $(y,s)$ to the parabolic boundary $\pa_p Q_1$.  By H\"{o}lder's inequality and \eqref{prleapl-3}, we  deduce from \eqref{prleapl-11} that
\begin{align} \label{prleapl-12}
\begin{split}
\int_{ Q_{1-2\de}} \big(|\na^2 u^\lm_0  |^2 +|\pa_t u^\lm_0  |^2\big)
 \leq  C \big(\de^{-1-\frac{2}{q}}+\|\na_x A^\lm \|^2_{\infty} +1\big) \Big\{\int_{Q_2}  | u^\lm_\va  |^2  +\int_{Q_2} |F|^2\Big\}.
\end{split}
\end{align}
By taking \eqref{prleapl-10} and \eqref{prleapl-12} into \eqref{prleapl-6},  we get
\begin{align}  \label{prleapl-13}
 \begin{split}
&\Big(\int_{Q_1}  \big| \na\big( u^\lm_\va -u^\lm_0  -\va S_\de ((\chi^\lm)^\va\na u^\lm_0  )\eta_\de  \big) \big|^2 \Big )^{1/2}\\
 &\leq  C \big(\de^{\frac{1}{2}-\frac{1}{q}} \!+\! \de \|\na_x A^\lm \|_{\infty}+\de^2 \|\pa_tA^\lm \|_\infty \big) \big(\de \|\na_x A^\lm \|_{\infty}\! +\!1 \big)\\
 & \quad\times \Big\{ \Big (\int_{Q_2}  | u^\lm_\va   |^2  \Big )^{1/2} + \Big (\int_{Q_2}  |F |^2 \Big )^{1/2} \Big\}.
 \end{split}
 \end{align}
Finally, note that by Lemma \ref{le2.0} and the second inequality in \eqref{prleapl-10},
 \begin{align*}
 \int_{Q_1} |\va S_\de ((\chi^\lm)^\va\na u^\lm_0  )\eta_\de |^2\leq C\va \int_{Q_1}  |\na u^\lm_0  |^2\leq C \va \Big\{\int_{Q_{2}} | u^\lm_\va |^2  +   \int_{Q_{2}} |F|^2 \Big\}.
 \end{align*}
We then derive \eqref{leap-3re} with $\sigma=\frac{1}{2}-\frac{1}{q}>0$ from \eqref{prleapl-13} and Poincar\'{e}'s inequality immediately.
\end{proof}

\begin{theorem}\label{thmap}
Assume that $A  $ satisfies conditions \eqref{cod1}, \eqref{cod2} and \eqref{cod3} for some $\theta\in (0,1]$ and $L>0.$ Let $\de=(1+\sqrt{\lm})\va$ and $u^\lm_\va$ be a solution to $\pa_t u^\lm_\va+\mathcal{L}^\lm_\va u^\lm_\va=F$ in $Q_{2r}$, where  $F \in L^2(Q_{2r})$ and $\de\leq r\leq 1$.  Then there exist a solution to
$ \pa_t u^\lm_0 +\mathcal{L}_0^\lm u^\lm_0  =F  $  in $ Q_{r}$,
such that
 \begin{align}\label{thmapre2}
 \begin{split}
 &\Big(\fint_{Q_{r }}  |  u^\lm_\va- u^\lm_0 |^2  \Big )^{1/2}
 \leq  C \big\{( \de r^{-1}  )^\sigma +\de^\theta L\big\} \left\{ \Big (\fint_{Q_{2r}}  | u^\lm_\va  |^2  \Big )^{1/2} + r^2\Big (\fint_{Q_{2r}}  |F |^2\Big )^{1/2} \right\},
 \end{split}
 \end{align}
where $0<\sigma<1$ and $C$ depend only on $d$ and $\mu$.
\end{theorem}
\begin{proof}
It suffices to consider the case $\de^\theta L < 1$, for otherwise the inequality is trivial.  We first consider the case $r=1.$
Let $S_\de $ be the smoothing operator defined as in \eqref{sva}. Denote $S_\de ((A^\lm)^\va)(x, t, x/\va, t/\va^2)$  as $ \widetilde{A^\lm } (x,t,x/\va,t/\va^2)$.
It is not difficult to see that
\begin{align}\label{leap4-1}
\|A^\lm - \widetilde{A^\lm } \|_{\infty}\leq C L \de^\theta, \quad \|\na_x \widetilde{A^\lm } \|_{\infty}\leq C L \de^{\theta-1} \quad\text{and}\quad \|\pa_t \widetilde{A^\lm } \|_{\infty}\leq C L \de^{ \theta-2}.
\end{align}
Let $v^\lm_\va$ be the weak solution to
\begin{align}\label{leap4-2}
\pa_t v^\lm_\va -\text{div} ( \widetilde{A^\lm } (x, t, x/\va, t/\va^2) \na v^\lm_\va)=F \quad \text{ in } Q_{3/2}
\end{align}
with $v^\lm_\va=u^\lm_\va$ on $\pa_p Q_{3/2}$. Note that
\begin{align*}
&\pa_t (u_\va^\lm-v_\va^\lm)-\text{div} \left\{ \widetilde{A^\lm } (x, t, x/\va, t/\va^2)\na (u_\va^\lm-  v^\lm_\va) \right\}\\
&= \text{div} \left\{[A^\lm (x, t, x/\va, t/\va^2)- \widetilde{A^\lm } (x, t, x/\va, t/\va^2) ] \na u_\va^\lm \right\}
\end{align*}
in $ Q_{3/2}.$ By Poincar\'{e}'s inequality  and standard energy estimates,
\begin{align}\label{leap4-3}
\begin{split}
 \int_{Q_{3/2}}| v^\lm_\va-u^\lm_\va |^2  &\leq C \int_{Q_{3/2}}|\na (v^\lm_\va-u^\lm_\va)|^2  \leq C  L^2 \de^{2\theta}   \int_{Q_{3/2}} |\na u^\lm_\va|^2 \\
&\leq C   L^2 \de^{2\theta}  \Big\{  \int_{Q_{2}} | u^\lm_\va|^2  +  \int_{Q_{2}} | F|^2 \Big\},
\end{split}
\end{align}
where Caccioppoli's inequality has been used for the last step.
Let $\pa_t-\text{div}(\widehat{ \widetilde{A^\lm } }(x,t ) \na )$ be the homogenized operator of $\pa_t - \text{div}( \widetilde{A^\lm } (x,t,x/\va,t/\va^2) \na )$, and  $v_0$ the solution to
\begin{align}
\pa_tv^\lm_0-\text{div}(\widehat{ \widetilde{A^\lm } }(x,t ) \na v^\lm_0 )=F \,\,\,\text{ in } Q_{5/4},\quad \quad v^\lm_0=v^\lm_\va \,\,\,\text{ on } \pa_p Q_{5/4}.
\nonumber\end{align}
 By Lemma \ref{leap-3}, \eqref{leap4-1} and \eqref{leap4-3},
 \begin{align}\label{leap4-4}
 \begin{split}
  \Big(\int_{Q_{5/4}}  | v^\lm_\va-v^\lm_0 |^2 \Big )^{1/2}
 &\leq C \big(\de^\sigma +\de^{\theta}L \big)
  \Big\{ \Big (\int_{Q_{3/2}}  | v^\lm_\va  |^2  \Big )^{1/2} + \Big (\int_{Q_{3/2}}  |F |^2 \Big )^{1/2} \Big\}\\
  &\leq C \big(\de^\sigma +\de^{\theta}L \big)
  \Big\{ \Big (\int_{Q_2}  | u^\lm_\va  |^2  \Big )^{1/2} + \Big (\int_{Q_2}  |F |^2 \Big )^{1/2} \Big\}.
 \end{split}
 \end{align}

Let $u^\lm_0$ be the solution to
\begin{align*}
\pa_tu^\lm_0-\text{div}(\widehat{A^\lm }(x,t) \na u^\lm_0 )=F \,\,\,\text{ in } Q_{1},\quad \quad u^\lm_0=v^\lm_0 \,\,\,\text{ on } \pa_p Q_{1}.
\end{align*}
 Since $\|A^\lm - \widetilde{A^\lm } \|_{\infty}\leq C L \de^\theta,$ similar to \eqref{lecore2} it is not difficult to verify that
$
\|\widehat{A^\lm }- \widehat{\widetilde{A^\lm }} \|_{\infty}\leq C L \de^\theta.
$
Therefore, as \eqref{leap4-3} by standard energy estimates and Caccioppoli's inequality, we deduce that
\begin{align}\label{leap4-6}
\begin{split}
 \int_{Q_1}| v^\lm_0-u^\lm_0 |^2 &\leq  C \int_{Q_1}|\na (v^\lm_0-u^\lm_0)|^2 \leq C L^2 \de^{2\theta}  \int_{Q_1} |\na v^\lm_0|^2
\\
& \leq C L^2 \de^{2\theta} \Big\{  \int_{Q_{5/4}} |v^\lm_0|^2  + \int_{Q_2} |F|^2 \Big\}\\
&\leq C L^2 \de^{2\theta}  \Big\{  \int_{Q_{2}} | u^\lm_\va|^2  +  \int_{Q_{2}} | F|^2\Big\},
\end{split}
\end{align} where  we have used  \eqref{leap4-4} for the last step.
In view of  \eqref{leap4-3}-\eqref{leap4-6}, we obtain \eqref{thmapre2} with $r=1$ immediately.

For general $\de\leq r\leq 1$, we set $\widetilde{u}_\va^\lm (x,t)=u^\lm_\va(rx,r^2t)$. Then $\widetilde{u}_\va^\lm $ is a solution to
$$\pa_t \widetilde{u}_\va^\lm -\text{div}(\mathcal{A^\lm }(x,t,x/{\va'},t/{\va'}^2)\na \widetilde{u}_\va^\lm )=\widetilde{F} \quad\text{in}\, \, Q_{2},$$ where $\mathcal{A^\lm }(x,t,y,s)=A^\lm (rx,r^2t,y, s), {\va'}=\va/r$ and $\widetilde{F}(x,t)=r^2 F(rx,r^2t).$ Observe that
\begin{align*}
\big|\mathcal{A^\lm }(x,t,y,s)- \mathcal{A^\lm }(x',t',y,s)\big|\leq r^\theta L\big\{|x-x'|+ |t-t'|^{1/2} \big\}^\theta.
\end{align*}
Therefore, there exists a solution $\widetilde{u}^\lm_0$ to $\pa_t \widetilde{u}^\lm_0-\text{div}(\widehat{\mathcal{A^\lm }}(x,t) \na \widetilde{u}^\lm_0)=\widetilde{F}$ in $Q_1$ such that
 \begin{align}\label{thmap-2}
 \begin{split}
  &\Big(\fint_{Q_1}  | \widetilde{u}_\va^\lm-\widetilde{u}^\lm_0  |^2 \Big )^{1/2}\\
 &\leq  C \big( (\de/r)^\sigma +  \de^\theta L \big)
  \left\{ \Big (\fint_{Q_2}  |  \widetilde{u}_\va^\lm  |^2  \Big )^{1/2} + \Big (\fint_{Q_2}  |\widetilde{F} |^2 \Big )^{1/2} \right\}.
 \end{split}
 \end{align}
Setting $u^\lm_0(x,t)=\widetilde{u}^\lm_0(x/r,t/r^2)$, then $u^\lm_0$ is a solution to  $ \pa_t u^\lm_0 +\mathcal{L}^\lm_0 u^\lm_0  =F  $  in $ Q_{r}$. By rescaling, one obtains  \eqref{thmapre2}  from \eqref{thmap-2} immediately. The proof is complete.
\end{proof}

\section{Interior Lipschitz estimates for $\pa_t+\mathcal{L}^\lm_\va$ }

This part is devoted to the uniform interior Lipschitz estimate for the operator  $\pa_t+\mathcal{L}^\lm_\va$ defined in \eqref{Llm}, i.e.,
$$\pa_t+\mathcal{L}^\lm_\va=\pa_t -\text{div}(A^\lm(x,t,x/\va,t/\va^2)\na).$$  Throughout the section, we always assume that  $A(x,t,y,s)$ satisfies \eqref{cod1}, \eqref{cod2} and \eqref{cod3} for some $\theta\in (0,1]$ and $L\geq 0$. Note that thanks to Lemma \ref{le-co}  the coefficient matrix  $\widehat{A^\lm}(x,t)$ of the homogenized operator $\pa_t +\mathcal{L}^\lm_0$ also satisfies \eqref{cod3}.

\begin{lemma}\label{inliple-1}
Let $u^\lm_0$ be a weak solution to $ \pa_t u^\lm_0+\mathcal{L}_0^\lm u^\lm_0 =F $ in $Q_r$, where $0<r\leq 1$, $F\in L^p(Q_r)$ for some $p>d+2$.
Define
\begin{align}\label{gru}
G(r;u^\lm_0)= \frac{1}{r}  \inf_{P\in \mathcal{P}}\Big\{\Big(\fint_{Q_r}|u^\lm_0-P|^2\Big)^{1/2}+r^{1+\nu}|\na  P|+ r^2 \Big(\fint_{Q_r}|F|^p \Big)^{1/p}\Big\},
\end{align}
where $\nu=\min\{\theta, 1-(d+2)/p\}$, and $\mathcal{P}$ denotes the set of linear functions  $E  x+b$ with $E\in \R^{d}$ and $b\in \R^1.$  Then
there exists $\zeta \in (0,1/4)$, depending only on $d$, $\mu$, $p$, and $(\theta,L)$ in \eqref{cod3}, such that
\begin{align}\label{inliple1re}
G(\zeta r;u^\lm_0)\leq \frac{1}{2} G(r;u^\lm_0).
\end{align}
\end{lemma}
\begin{proof}
Taking $P_0=\na u^\lm_0(0,0) x+u^\lm_0(0,0)$, we deduce that
\begin{align}\label{prinliple-1}
 G(\zeta r;u^\lm_0)
 & \leq  (\zeta r)^{\nu}\left\{\|\na u^\lm_0\|_{C_x^{\nu}(Q_{\zeta r})}\!+\!\|u^\lm_0\|_{C_t^{(1+\nu)/2  }(Q_{\zeta r})} \right\} \! +\!\zeta r \Big(\fint_{Q_{\zeta r}}|F|^p \Big)^{1/p} \!+\!(\zeta r)^\nu  |\na u^\lm_0(0,0)|\nonumber\\
 &\leq  (\zeta r)^{\nu}\left\{\|\na (u^\lm_0-P)\|_{C_x^{\nu}(Q_{\zeta r})}+ \|u^\lm_0-P\|_{C_t^{(1+\nu)/2}(Q_{\zeta r})} \right\} \nonumber\\
 &\quad +r\zeta^{1-(d+2)/p} \Big(\fint_{Q_r}|F|^p \Big)^{1/p} +(\zeta r)^\nu |\na u^\lm_0(0,0)|,
\end{align}
where $\|\cdot\|_{C_t^{\al}(Q_{r})}(0<\al<1)$ is given by \eqref{htseminorm}, and $\|\cdot\|_{C_x^{\al}(Q_{r})}$ is defined as following,
\begin{align*}
&\|u\|_{C_x^\al(Q_r)}=\sup_{\substack{(x,t), (y,t)\in Q_r\\ x\neq y}} \frac{|u(x,t)-u(y,t)|}{|x-y|^\al}.
\end{align*}
Since $\widehat{A^\lm}$ satisfies \eqref{cod3}. The interior $C^{1,1/2}$ estimate of $\pa_t +\mathcal{L}_0^\lm$ gives
\begin{align}\label{prinliple-2}
\begin{split}
 |\na u^\lm_0(0,0)|&\leq \frac{C}{r} \Big(\fint_{Q_r}|u^\lm_0-P(0)|^2\Big)^{1/2} +Cr \Big(\fint_{Q_r}|F|^p\Big)^{1/p}\\
  &\leq  \frac{C}{r} \Big(\fint_{Q_r}|u^\lm_0-P|^2\Big)^{1/2} +C|\na P| +Cr \Big(\fint_{Q_r}|F|^p\Big)^{1/p}.
 \end{split}
\end{align}
Furthermore,  since
$$ \pa_t(u^\lm_0-P)- \text{div} \big(\widehat{A^\lm}(x,t) \na (u^\lm_0-P)\big)=F+\text{div} \big( (\widehat{A^\lm}(x,t)- \widehat{A^\lm}(0,0)) \na P  \big) \quad \text{in }\, Q_r.$$
 Schauder estimates of $\pa_t +\mathcal{L}^\lm_0$ imply for $0<\zeta<1/2$,
\begin{align}\label{prinliple-3}
\begin{split}
&\|\na (u^\lm_0-P)\|_{C_x^{\nu}(Q_{\zeta r})}+\|u^\lm_0-P\|_{C_t^{(1+\nu)/2}(Q_{\zeta r})}\\
 &\leq \|\na (u^\lm_0-P)\|_{C^{\nu,\nu/2}(Q_{r/2})}+\|u^\lm_0-P\|_{C_t^{(1+\nu)/2}(Q_{r/2})}\\
 &\leq \frac{C}{r^{1+\nu}} \Big(\fint_{Q_r} |u^\lm_0-P|^2\Big)^{1/2}+ C r^{-\nu} \|(\widehat{A^\lm}-\widehat{A^\lm}(0,0)) \na P\|_{L^\infty(Q_r)}\\
 &\quad+ C \|(\widehat{A^\lm}-\widehat{A^\lm}(0,0)) \na P\|_{C^{\nu,\nu/2}(Q_r)}+Cr^{1-\nu}\Big(\fint_{Q_r}|F|^p\Big)^{1/p}\\
 &\leq  \frac{C}{r^{1+\nu}} \Big(\fint_{Q_r} |u^\lm_0-P|^2\Big)^{1/2} +C|\na P|+ Cr^{1-\nu}\Big(\fint_{Q_r}|F|^p\Big)^{1/p}.
 \end{split}
\end{align}
By taking \eqref{prinliple-2} and \eqref{prinliple-3} into  \eqref{prinliple-1}, we get \eqref{inliple1re} immediately for $\zeta$ small enough.
\end{proof}

\begin{lemma}\label{inliple-2}
Let $u^\lm_\va$ be a solution to $ \pa_t u^\lm_\va-\text{div}(A^\lm (x,t,x/\va,t/\va^2)\na u^\lm_\va)=F $ in $Q_1$ with $F\in L^p(Q_1), p> d+2$ and $0<\va<1.$
Let $\de=(1+\sqrt{\lm})\va$, and  $\zeta \in (0,1/4)$ be given by Lemma \ref{inliple-1}. Then for any $\de\leq r\leq 1$, we have
\begin{align}\label{inliple-2re}
G(\zeta r; u^\lm_\va)\leq \frac{1}{2} G(r; u^\lm_\va) +C  \Big(\frac{\de}{r} \Big)^\sigma     \left\{\frac{1}{r} \Big (\fint_{Q_{2r}}  | u^\lm_\va  |^2  \Big )^{1/2} + r\Big (\fint_{Q_{2r}}  |F |^p\Big )^{1/p} \right\}
\end{align}
for some $\sigma$ depending only on  $d$, $\mu$, and $\theta$ in \eqref{cod3}, where $C$ depends only on $d$, $\mu$, $p$, and $(\theta,L)$ in \eqref{cod3}.
\end{lemma}
\begin{proof}
The proof is the same as that of Lemma 8.4 in \cite{shenan2017}. By Lemma \ref{inliple-1},
\begin{align*}
G(\zeta r; u^\lm_\va)&\leq G(\zeta r; u^\lm_0) +\frac{C}{\zeta r} \Big(\fint_{Q_r} |u^\lm_\va-u^\lm_0|^2\Big)^{1/2}\\
&\leq \frac{1}{2} G(r; u^\lm_0) +\frac{C}{\zeta r} \Big(\fint_{Q_r} |u^\lm_\va-u^\lm_0|^2\Big)^{1/2}\\
&\leq  \frac{1}{2} G(r; u^\lm_\va) +\frac{C(\zeta^{-1}+1)}{r} \Big(\fint_{Q_r} |u^\lm_\va-u^\lm_0|^2\Big)^{1/2},
\end{align*}
 from which and  Theorem \ref{thmap}, we obtain \eqref{inliple-2re} immediately.
\end{proof}

Based on Lemmas \ref{inliple-1} and \ref{inliple-2}, we give the interior Lipschitz estimate uniform down to the scale $ (1+\sqrt{\lm})\va$ for the operator $\pa_t+\mathcal{L}^\lm_\va$.
\begin{theorem}\label{thmllip}
 Let $u^\lm_\va $ be a solution to $ \pa_t u^\lm_\va -\text{div}(A^\lm(x,t,x/\va,t/\va^2)\na u^\lm_\va )=F $ in $Q_1$ with $F\in L^p(Q_1), p> d+2$. Then for any $ (1+\sqrt{\lm}) \va \leq r<1$, we have
 \begin{align}\label{thmllipre}
 \Big(\fint_{Q_r}|\na u^\lm_\va |^2\Big)^{1/2}\leq C \Big\{\Big(\fint_{Q_1}|\na u^\lm_\va |^2\Big)^{1/2} +\Big(\fint_{Q_1}|F|^p\Big)^{1/p} \Big\},
 \end{align}
 where $C$ depends only on $d$, $\mu$, $p$ and $(\theta,L)$ in \eqref{cod3}.
\end{theorem}
\begin{proof}
The proof,  based on Lemma \ref{inliple-2} and the abstract iteration lemma in \cite{shenan2017} (Lemma 8.5), is almost the same as that for Theorem 6.1 in \cite{gsamar2020}, where the Lipschitz estimate uniform down to the scale $\va+\va^{\ell/2}$ for the operator $\pa_t -\text{div}(A(x/\va, t/\va^{\ell})\na)$ was established. Let us omit the details here for concision.
\end{proof}

 For case $\kappa=\va$, Theorem \ref{thmllip} gives the  interior Lipschitz estimate for the locally periodic  operator  $\pa_t  +\mathcal{L}_\va =\pa_t - \text{div} (A(x,t,x/\va,t/\va^2) \na )$ uniform down to the scale $\va$. Precisely speaking, let $u_\va$ be a weak solution to $ \pa_t u_\va +\mathcal{L}_\va u_\va  =F $ in $Q_1$ with $F\in L^p(Q_1), p> d+2$. Then for any $\va\leq r<1$,
 \begin{align}\label{thmllipre'}
 \Big(\fint_{Q_r}|\na u_\va |^2\Big)^{1/2}\leq C \Big\{\Big(\fint_{Q_1}|\na u_\va |^2\Big)^{1/2} +\Big(\fint_{Q_1}|F|^p\Big)^{1/p} \Big\},
 \end{align} where $C$ depends only on $d$, $\mu$, $p$ and $(\theta, L)$ in \eqref{cod3}.
Under additional smoothness assumption on $A$, one can get the uniform Lipschitz estimate in the small scales for the operator $\pa_t+\mathcal{L}_\va$ , which is new to our best knowledge.
\begin{theorem}\label{pointlip}
 Assume that $A(x,t,y,s)$ satisfies \eqref{cod1}, \eqref{cod2} and \eqref{cod33}.
 Let $u_\va $ be a solution to $ \pa_t u_\va -\text{div}(A(x,t,x/\va,t/\va^2)\na u_\va )=F $ in $Q_1$ with $F\in L^p(Q_1), p> d+2$. Then for any $0<\va<\infty$,
 \begin{align}\label{pointlipre}
|\na u_\va (0,0)| \leq C \Big\{\Big(\fint_{Q_1}|\na u_\va |^2\Big)^{1/2} +\Big(\fint_{Q_1}|F|^p\Big)^{1/p} \Big\},
 \end{align}
 where $C$ depends only on $d$, $\mu$, $p$, and $(\vartheta,M)$ in \eqref{cod33}.
\end{theorem}
\begin{proof}
It suffices to consider the case $0<\va<1/4$, for otherwise $A(x,t,x/\va,t/\va^2)$ is uniformly H\"{o}lder continuous in $(x,t).$ Then the estimate \eqref{pointlipre} follows from the well-known Schauder estimates of parabolic systems with H\"{o}lder continuous coefficients.

To handle the case $0<\va<1/4$, we set $w_\va=\va^{-1}u(\va x,\va^2 t)$.  It is obvious that $\na w_\va(0,0)=\na u_\va (0,0),$ and $w_\va$ satisfies the equation
\begin{align*}
\pa_t w_\va -\text{div} (A(\va x, \va^2t, x, t ) \na w_\va)=\va F(\va x,\va^2t).
\end{align*}
Since $ A(\va x, \va^2 t, x,t)$ is uniformly H\"{o}lder continuous in $(x,t)$. By the interior Lipschitz estimates of parabolic systems with H\"{o}lder continuous coefficients, we have
\begin{align*}
|\na w_\va(0,0)|&\leq C \Big\{\Big(\fint_{Q_1}|\na w_\va|^2\Big)^{1/2} +\va \Big(\fint_{Q_1}|F(\va x, \va^2 t)|^p\Big)^{1/p} \Big\}\\
&\leq C  \Big\{\Big(\fint_{Q_\va}|\na u_\va |^2\Big)^{1/2} + \va^{1-(d+2)/p}\Big(\fint_{Q_1}|F(x,t)|^p\Big)^{1/p} \Big\},
\end{align*}
which, together with \eqref{thmllipre'}, gives \eqref{pointlipre}.
\end{proof}

As direct consequences of Theorem \ref{pointlip}, we have the following two corollaries, which provide the Lipschitz estimates for the locally periodic  parabolic operators with only spatial or temporal oscillations.
\begin{coro}\label{coro1}
Let $v_\va$ be a weak solution to $\pa_t v_\va -\text{div} (A(x,t,t/\va) \na v_\va)=F$ in $Q_1$, where $F\in L^p(Q_1)$ with $p>d+2$,   $A=A(x,t,s)$ is 1-periodic in $s$ and satisfies \eqref{cod1}. Assume  that
\begin{align}\label{coro1con}
|A(x,t,s)-A(x',t',s')|\leq L \big(|x-x'|+|t-t'|^{1/2}+|s-s'|^{1/2} \big)^\theta
\end{align} for some  $\theta\in (0,1]$ and $L>0$.
Then for any $0<\va <\infty,$
\begin{align*}
|\na v_\va(0,0)| \leq C \Big\{\Big(\fint_{Q_1}|\na v_\va|^2\Big)^{1/2} +\Big(\fint_{Q_1}|F|^p\Big)^{1/p} \Big\},
 \end{align*}
 where $C$ depends only on $d$, $\mu$, $p,$ and $(\theta, L)$ in \eqref{coro1con}.
\end{coro}

\begin{coro}\label{coro2}
Let $v_\va$ be a weak solution to $\pa_t v_\va -\text{div} (A(x,t,x/\va) \na v_\va)=F$ in $Q_1$, where $F\in L^p(Q_1)$ with $p>d+2$,   $A=A(x,t,y)$ is 1-periodic in $y$ and satisfies \eqref{cod1}. Assume that
\begin{align}\label{coro2con}
|A(x,t,y)-A(x',t',y')|\leq L \big(|x-x'|+|y-y'|+|t-t'|^{1/2} \big)^\theta
\end{align}for some  $\theta\in (0,1]$ and $L>0$.
Then for any $0<\va <\infty,$
\begin{align}\label{coro2re}
|\na v_\va(0,0)| \leq C \Big\{\Big(\fint_{Q_1}|\na v_\va|^2\Big)^{1/2} +\Big(\fint_{Q_1}|F|^p\Big)^{1/p} \Big\},
 \end{align} where $C$ depends only on $d$, $\mu$, $p$, and $(\theta, L)$ in \eqref{coro2con}.
\end{coro}

\section{Boundary Lipschitz estimates for $\pa_t+\mathcal{L}^\lm_\va$ }
Now let us consider the large-scale boundary Lipschitz estimates for the operator $\pa_t+\mathcal{L}^\lm_\va$ in \eqref{Llm}.
As in Section 4, if not stated $A(x,t,y,s)$ is always assumed to satisfy \eqref{cod1}, \eqref{cod2},  and \eqref{cod3} for some $0<\theta\leq 1$ and $L>0$.

 Let $\psi: \R^{d-1} \rightarrow \R$  be a $C^{1,\al} (0<\al<1)$ function such that $\psi(0)=0$ and $\|\psi\|_{C^{1,\al}(\R^{d-1})}\leq M_0$.
 For $0<r<\infty$, set
 \begin{align}\label{deboundary}
 \begin{split}
 &Z_r=Z(r,\psi)=\{(x',x_d):|x'|<r \text{ and } \psi(x')<x_d<\psi(x')+10(M_0\!+\!1)r\} \times (-r^2, 0),\\
 &I_r=I(r,\psi)=\{(x',x_d):|x'|<r \text{ and } x_d=\psi(x')\} \times (-r^2, 0).
 \end{split}
 \end{align}
Let $C^{1+\al}(I_{r})$ be the parabolic H\"{o}lder space with the scale-invariant norm
 \begin{align}
 \| f\|_{C^{1+\al}(I_{r})}=\|f\|_{L^\infty(I_r)}+r\|\na_{tan} f\|_{L^\infty(I_r)}+ r^{1+\al}\|\na_{tan} f\|_{C^{\al,\frac{\al}{2}}(I_r)}+r^{1+\al} \|f\|_{C_t^{\frac{1+\al}{2}}(I_r)}.
 \end{align}

%

 \begin{theorem}\label{b-thmap}
 Let $u^\lm_\va$ be a solution to
 \begin{align*}
 \pa_t u^\lm_\va+\mathcal{L}^\lm_\va u^\lm_\va=F ~\text{ in }Z_{2r}, \quad \text{ and } \quad u^\lm_\va=g ~\text{ on } I_{2r}
 \end{align*}
  with $(1+\sqrt{\lm})\va=\de\leq r\leq 1$, and $F \in L^2(Z_{2r}), g\in C^{1+\al}(I_{2r})$.  Then there exist a solution to
$ \pa_t u^\lm_0 +\mathcal{L}^\lm_0 u^\lm_0  =F  $  in $ Z_{r}$,
such that
 \begin{align}\label{b-thmapre2}
 \begin{split}
 &\Big(\fint_{Z_{r }} |  u^\lm_\va- u^\lm_0 |^2  \Big )^{1/2}
 \leq  C \Big(\Big(\frac{\de}{r}\Big)^\sigma +\de^\theta L\Big) \left\{ \Big (\fint_{Z_{2r}}  | u^\lm_\va  |^2  \Big )^{1/2}+ r^2\Big (\fint_{Z_{2r}}  |F |^2\Big )^{1/2} + \|g\|_{C^{1+\al}(I_{2r})}\right\},
 \end{split}
 \end{align}
where $0<\sigma<1$  and $C$ depend  only on $d$, $\mu$, $\al$, and $M_0 $.
 \end{theorem}
\begin{proof}
 The proof is completely parallel to that of Theorem  \ref{thmap}. We just sketch it. By rescaling, it suffices to consider the case $r=1.$\\
\noindent{\textbf{Step 1.}} Assume that $A=A(x,t,y,s)$ satisfies (\ref{cod1}), (\ref{cod2}) and \eqref{ainfty}.  Let $u^\lm_0$ be the solution to $ \pa_t u^\lm_0 +\mathcal{L}^\lm_0 u^\lm_0  =F  $  in $ Z_1$, and $u^\lm_0=u^\lm_\va $ on $\pa_p Z_1$, where $\pa_t+\mathcal{L}^\lm_0$ is the homogenized operator of $\pa_t+\mathcal{L}^\lm_\va$.
Let $w^\lm_\va $ be defined as in \eqref{def-w} with the cut-off function $\eta_\de \in C_c^\infty(\R^{d+1})$,
  $0\leq\eta_\de\leq 1$ and
\begin{align*}
 \begin{split}
 &\eta_\de =1  \,\,\text{ in } Z_{1-6\de},\quad  \eta_\de =0 \,\,\text{ in }  Z_1\!\setminus\! Z_{1-4\de}, \\
& |\na \eta_\de |\leq C \de^{-1}, \quad |\pa_t \eta_\de |+ |\nabla^2\eta_\de|\leq C \de^{-2}.
\end{split}\end{align*}
By performing similar analysis as in Lemma \ref{leap-2}, we can prove that
 \begin{align*}
\|\na w^\lm_\va \|_{L^2(Z_1)}
&\leq C \big\{\de\|\na_x A^\lm \|_{\infty} +1 \big\}  \|\na u^\lm_0\|_{L^2(Z_1\setminus Z_{1-5\de})}\nonumber\\
 &\quad +C \de\big(\|\na_x A^\lm\|_{\infty}\!+\!\de \|\pa_t A\|_{\infty}\big)\|\na u^\lm_0\|_{L^2(Z_1)} \nonumber\\
&\quad+ C\de \big (\de\|\na_x A^\lm\|_{\infty} +1 \big)
  \big (\|\na^2 u^\lm_0\|_{L^2(Z_{1-2\de})} +\|\pa_t u^\lm_0\|_{L^2(Z_{1-2\de})}\big),
\end{align*}
which, together with Poincar\'{e}'s inequality and estimates \eqref{prleapl-5} and \eqref{prleapl-5'} on $Z_1$ instead of $Q_1$, implies that
\begin{align*}
\|u^\lm_\va -u^\lm_0\|_{L^2(Z_1)}
&\leq C  \big\{\de\|\na_x A^\lm \|_{\infty} +1 \big\}\|\na u^\lm_0\|_{L^2(Z_1)\setminus Z_{1-5\de})}
 \nonumber\\
&\quad+C \de\big(\|\na_x A^\lm\|_{\infty}\!+\!\de \|\pa_t A^\lm\|_{\infty}\big)\|\na u^\lm_0\|_{L^2(Z_1)}\\
&\quad+ C\de \big (\de\|\na_x A^\lm\|_{\infty} +1 \big)
  \big (\|\na^2 u^\lm_0\|_{L^2(Z_{1-2\de})} +\|\pa_t u^\lm_0\|_{L^2(Z_{1-2\de})}\big).\nonumber
\end{align*}
By the Meyer's estimate, the interior $H^2$ estimate, and Caccioppolli's inequality, we can deduce that
\begin{align}\label{prb-thmapre2-1}
\begin{split}
  \| u^\lm_\va -u^\lm_0\|_{L^2(Z_{1}) }
 &\leq  C \big(\de^\sigma \!+\! \de \|\na_x A^\lm\|_{\infty}+ \de^2\|\pa_t A^\lm\|_{\infty}\big) \big(\de \|\na_x A^\lm\|_{\infty}\! +\!1 \big)\\
 & \quad\times \big\{  \|u^\lm_\va \|_{L^2(Z_2)} +  \|F\|_{L^2(Z_2)}+ \|g\|_{C^{1+\al}(I_2)} \big\}
\end{split}
\end{align}
 for some $0<\sigma<1$, where $C$ depends only on $d, \mu, \al$  and $M_0 $.

\noindent{\textbf{Step 2.}} Now we assume that $A=A(x,t,y,s)$ satisfies (\ref{cod1})-\eqref{cod3}. By approximating $A$ with matrices smooth in the $x,t$ variables as in Theorem \ref{thmap}, one can prove the desired estimate accordingly. We omit the details and close the proof here.
\end{proof}

\begin{lemma}\label{b-liple-1}
Let $u^\lm_0$ be a weak solution to $ \pa_t u^\lm_0+\mathcal{L}^\lm_0 u^\lm_0 =F $ in $Z_r$, and $u^\lm_0=g$ on $I_r$, where $0<r\leq 1$, $F\in L^p(Z_r)$ for some $p>d+2$ and $g\in C^{1+\al}(I_r)$.
Define
\begin{align}\label{b-gru}
\mathfrak{G}(r;u^\lm_0)= \frac{1}{r}\inf_{P\in \mathcal{P}} \Big\{\Big(\fint_{Z_r}|u^\lm_0\!-\!P|^2\Big)^{1/2}+\|g\!-\!P\|_{C^{1+\al}(I_r)}\Big\}+r^{\nu}|\na  P|+ r  \Big(\fint_{Z_r}|F|^p \Big)^{1/p},
\end{align}
where $\nu=\min\{\theta, \al, 1-(d+2)/p\}$, and $\mathcal{P}$ denotes the set of linear functions  $E  x+b$ with $E\in \R^{d}$ and $b\in \R.$  Then
there exists $\tau\in (0,1/4)$, depending only on $d$, $\mu$, $p$, $\al$, $M_0,$ and $(\theta,L)$ in \eqref{cod3}, such that
\begin{align}\label{b-liple1re}
\mathfrak{G}(\tau r;u^\lm_0)\leq \frac{1}{2} \mathfrak{G}(r;u^\lm_0).
\end{align}
\end{lemma}
\begin{proof}
Take $P_0=\na u^\lm_0(0,0) x+u^\lm_0(0,0)$. Similar to \eqref{prinliple-1} we deduce that
\begin{align}\label{prb-liple-1}
 \mathfrak{G}(\tau r;u^\lm_0) &\leq \frac{1}{\tau r}\|u^\lm_0\!-\!P_0\|_{L^\infty(Z_{\tau r})}\!+\!\frac{1}{\tau r}\|g\!-\!P_0\|_{C^{1+\al}(I_{\tau r})}\!+\!(\tau r)^\nu  |\na u^\lm_0(0,0)|\!+\!\tau r  \Big(\fint_{Z_{\tau r}}|F|^p \Big)^{1/p} \nonumber\\
 &\leq  (\tau r)^{\nu}\Big(\|\na (u^\lm_0\!-\!P)\|_{C^{\nu,\nu/2}(Z_{\tau r})}+ \|u^\lm_0\!-\!P\|_{C_t^{(1+\nu)/2}(Z_{\tau r})} \Big) + \frac{\tau^\al}{r}    \|g\!-\!P\|_{C^{1+\al}(I_{r})}\nonumber\\
 &\quad  + (\tau r)^\nu  |\na u^\lm_0(0,0)\!-\!\na P|+(\tau r)^\nu |\na P|+r\tau^{1-(d+2)/p} \Big(\fint_{Z_r}|F|^p \Big)^{1/p}.
\end{align}
Since $u^\lm_0-P=g-P$ on $I_r$ and
$$ \pa_t(u^\lm_0-P)- \text{div} \big(\widehat{A^\lm}(x,t) \na (u^\lm_0-P)\big)=F+\text{div} \big( (\widehat{A^\lm}(x,t)- \widehat{A^\lm}(0,0)) \na P  \big) \quad \text{in }\, Z_r.$$
By using  Schauder estimates of the operator $\pa_t +\mathcal{L}^\lm_0$, we have for $0<\tau<1/2$,
\begin{align*}
\begin{split}
&\|\na (u^\lm_0-P)\|_{C^{\nu,\nu/2}(Z_{\tau r})}+\|u^\lm_0-P\|_{C_t^{(1+\nu)/2}(Z_{\tau r})}+r^{-\nu}\|\na (u^\lm_0-P)\|_{L^\infty(Z_{\tau r})}\\
 &\leq \|\na (u^\lm_0-P)\|_{C^{\nu,\nu/2}(Z_{r/2})}+\|u^\lm_0-P\|_{C_t^{(1+\nu)/2}(Z_{r/2})}+r^{-\nu}\|\na (u^\lm_0-P)\|_{L^\infty(Z_{r/2})}\\
 &\leq \frac{C}{r^{1+\nu}} \Big(\fint_{Z_r} |u^\lm_0-P|^2\Big)^{1/2}+ C r^{-\nu} \|(\widehat{A^\lm}-\widehat{A^\lm}(0,0)) \na P\|_{L^\infty(Z_r)}\\
 &\quad+ C \|(\widehat{A^\lm}-\widehat{A^\lm}(0,0)) \na P\|_{C^{\nu,\nu/2}(Z_r)}+Cr^{1-\nu}\Big(\fint_{Z_r}|F|^p\Big)^{1/p}+\frac{1}{  r^{1+\nu}}\|g-P\||_{C^{1+\al}(I_r)}\\
 &\leq  \frac{C}{r^{1+\nu}} \Big(\fint_{Z_r} |u^\lm_0-P|^2\Big)^{1/2} +C|\na P|+ Cr^{1-\nu}\Big(\fint_{Z_r}|F|^p\Big)^{1/p}+\frac{1}{  r^{1+\nu}}\|g-P\||_{C^{1+\al}(I_r)},
 \end{split}
\end{align*}
from which and \eqref{prb-liple-1},  the desired result follows for $\tau$ small enough.
\end{proof}
\begin{lemma}\label{b-liple-2}
Let $u^\lm_\va $ be a weak solution to $ \pa_t u^\lm_\va +\mathcal{L}^\lm_\va u^\lm_\va  =F $ in $Z_r$, and $u^\lm_\va =g$ on $I_r$ where $0<r\leq 1$, $F\in L^p(Z_r)$ for some $p>d+2$ and $g\in C^{1+\al}(I_r)$.
Let $\tau \in (0,1/4)$ be given by Lemma \ref{b-liple-1}. Then for any $(1+\sqrt{\lm})\va=\de\leq r\leq 1$, we have
\begin{align}\label{b-liple-2re}
\mathfrak{G}(\tau r; u^\lm_\va )\leq \frac{1}{2} \mathfrak{G}(r; u^\lm_\va ) +C \Big(\frac{\de}{r}\Big)^\sigma \frac{1}{r}  \Big\{\Big (\fint_{Z_{r}}  | u^\lm_\va   |^2  \Big )^{1/2} + r^2\Big (\fint_{Z_{r}}  |F |^p\Big )^{1/p}+ \|g\|_{C^{1+\al}(I_r)} \Big\}
\end{align}
for some $\sigma$ depending only on  $d$, $\mu$, and $\theta$ in \eqref{cod3}, where $C$ depends only on $d$, $\mu$, $p$, $\al$, $M_0$, and $(\theta,L)$ in \eqref{cod3}.
\end{lemma}
\begin{proof}
The proof, which relies on Theorem \ref{b-thmap} and Lemma \ref{b-liple-1}, is completely the same as \eqref{inliple-2}.
\end{proof}


 \begin{theorem}\label{b-thmllip}
 Let $u^\lm_\va$ be a solution to $ \pa_t u^\lm_\va+\mathcal{L}^\lm_\va  u^\lm_\va=F $ in $Z_1$, and $u^\lm_\va=g$ on $I_1$, where $F\in L^p(Z_1)$ with $p> d+2$, and $g\in C^{1+\al} (I_1)$. Then for any $0<(1+\sqrt{\lm})\va\leq r<1$, we have
 \begin{align}\label{b-thmllipre}
 \Big(\fint_{Z_r}|\na u^\lm_\va|^2\Big)^{1/2}\leq C \Big\{\Big(\fint_{Z_1}\na u^\lm_\va|^2\Big)^{1/2} +\Big(\fint_{Z_1}|F|^p\Big)^{1/p}+\|g\|_{C^{1+\al}(I_1)} \Big\},
 \end{align}
 where $C$ depends only on $d$, $\mu$, $p$, $\al$, $M_0$ and $(\theta,L)$ in \eqref{cod3}.
\end{theorem}
\begin{proof}
Parallel to Theorem \ref{thmllip}, the proof is based on Lemma \ref{b-liple-2} and  Lemma 8.5 in \cite{shenan2017}. Let us  omit the details.
\end{proof}

In the case $\kappa=\va$, we can take $\lm=1$ in Theorem \ref{b-thmllip} and derive the boundary Lipschitz estimate for $\pa_t+\mathcal{L}_\va=\pa_t-\text{div} (A(x,t,x/\va,t/\va^2)\na )$ uniform down to scale $\va$. Indeed, let $u_\va$ be a weak solution to $ \pa_t u_\va +\mathcal{L}_\va u_\va  =F $ in $Z_1$, and $u_\va=g$ on $I_1$, where $F\in L^p(Q_1), p> d+2$, and $g\in C^{1+\al}(I_1), 1<\al<1$. Then, similar to \eqref{thmllipre'}, we have
\begin{align}\label{b-thmllipre}
 \Big(\fint_{Z_r}|\na u_\va|^2\Big)^{1/2}\leq C \Big\{\Big(\fint_{Z_1}\na u_\va|^2\Big)^{1/2} +\Big(\fint_{Z_1}|F|^p\Big)^{1/p}+\|g\|_{C^{1+\al}(I_1)} \Big\}.
 \end{align} for any $\va\leq r<1$.
 Under additional smoothness assumption on $A$, one can derive the uniform boundary Lipschitz estimates in the small scales for $\pa_t+\mathcal{L}_\va$.
\begin{theorem} \label{b-lip}
Assume that $A(x,t,y,s)$ satisfies \eqref{cod1}, \eqref{cod2} and \eqref{cod33}.
Let $u_\va$ be a solution to $ \pa_t u_\va-\text{div}(A(x,t,x/\va,t/\va^2)\na u_\va)=F $ in $Z_1$, and $u_\va=g$ on $I_1$, where $F\in L^p(Z_1)$ with $p> d+2$, and $g\in C^{1+\al} (I_1)$.
Then for any $0<\va<\infty$,
 \begin{align}\label{b-pointlipre}
|\na u_\va(0,0)| \leq C \Big\{\Big(\fint_{Z_1}|\na u_\va|^2\Big)^{1/2} +\Big(\fint_{Z_1}|F|^p\Big)^{1/p} +\|g\|_{C^{1+\al}(I_1)}\Big\},
 \end{align}
 where $C$ depends only on $d$, $\mu$, $p$, $\al$, $M_0$, and $(\vartheta,M)$ in \eqref{cod33}.
\end{theorem}
\begin{proof}
The proof, based on standard blow-up argument, is the same as that for Theorem \ref{pointlip}. We therefore omit the details.
\end{proof}

As direct consequences of Theorem \ref{b-lip}, the following two corollaries are completely parallel to Corollaries  \ref{coro1} and \ref{coro2}.
\begin{coro}\label{b-coro1} Let $A=A(x,t,s)$ be 1-periodic in $s$ and satisfy \eqref{cod1} and \eqref{coro1con}.
Let $v_\va$ be a weak solution to $\pa_t v_\va -\text{div} (A(x,t,t/\va) \na v_\va)=F$ in $Z_1$, and $v_\va=g$ on $I_1$, where $F\in L^p(Z_1), p> d+2$ , and $g\in C^{1+\al} (I_1)$.
Then for any $0<\va <\infty,$
\begin{align*}
|\na v_\va(0,0)| \leq C \Big\{\Big(\fint_{Z_1}|\na v_\va|^2\Big)^{1/2} +\Big(\fint_{Z_1}|F|^p\Big)^{1/p}+\|g\|_{C^{1+\al}(I_1)} \Big\},
 \end{align*}
 where $C$ depends only on $d$, $\mu$, $p$, $\al$, $M_0$, and  $(\theta, L)$ in \eqref{coro1con}.
\end{coro}

\begin{coro}\label{b-coro2} Let $A=A(x,t,y)$ be 1-periodic in $y$ and satisfy \eqref{cod1} and \eqref{coro2con}.
Let $v_\va$ be a weak solution to $\pa_t v_\va -\text{div} (A(x,t,x/\va) \na v_\va)=F$ in $Z_1$, and $v_\va=g$ on $I_1$, where $F\in L^p(Z_1)$ for some $p> d+2$, and $g\in C^{1+\al} (I_1)$.
Then for any $0<\va <\infty,$
\begin{align}\label{b-coro2re}
|\na v_\va(0,0)| \leq C \Big\{\Big(\fint_{Z_1}|\na v_\va|^2\Big)^{1/2} +\Big(\fint_{Z_1}|F|^p\Big)^{1/p}+\|g\|_{C^{1+\al}(I_1)} \Big\},
 \end{align} where $C$ depends only on $d$, $\mu$, $p$, $\al$, $M_0$, and $(\theta, L)$ in \eqref{coro2con}.
\end{coro}

\section{ Proof of Theorems \ref{thm1} and \ref{b-thm2}}

We are now ready to prove the uniform Lipschitz estimates for the operator $$\pa_t+\mathfrak{L}_\va=\pa_t -\text{div}(A(x,t,x/\va,t/\kappa^2) \na), \quad \quad 0< \va,\kappa<1.$$

\begin{proof}[\textbf{Proof of Theorem 1.1}] By translation, it suffices to consider the case $(x_0,t_0)=(0,0).$
Let $u_\va$ be a weak solution of $\pa_t u_\va+\mathfrak{L}_\va u_\va=F$ in $Q_1=Q_1(0,0)$  with $F\in L^p(Q_1 ), p>d+2.$ Note that $\pa_t+\mathfrak{L}_\va=\pa_t +\mathcal{L}_\va^\lm $ and $(1+\sqrt{\lm})\va=\va+\kappa$ when $\lm=\kappa^2/\va^2$.
By taking $\lm=\kappa^2/\va^2$ in Theorem \ref{thmllip}, it follows that for any $ \va+\kappa \leq r<  1,$
\begin{align}\label{pthm1-1}
\Big(\fint_{Q_r }|\na u_\va|^2\Big)^{1/2}\leq C \Big\{\Big(\fint_{Q_1}|\na u_\va|^2\Big)^{1/2} + \Big(\fint_{Q_1 }|F|^p\Big)^{1/p} \Big\},
\end{align}
which is exactly \eqref{thm1-re1}.

To verify \eqref{thm1-re2}, we may assume that  $0<\va,\kappa<c_0$ for some small constant $c_0\in (0,1)$, for otherwise we have $\va\geq c_0, \kappa \geq c_0$; $\va\geq c_0, \kappa<c_0$; or $\kappa\geq c_0, \va<c_0$.
For the first case, the coefficient $A(x, t, x/\va, t/\ka^2)$ is uniformly H\"{o}lder continuous in $(x,t)$. One can derive \eqref{thm1-re2} from the standard Lipschitz estimates for parabolic equations with H\"{o}lder continuous coefficients. For the other two cases, the operator $\pa_t+\mathfrak{L}_\va$ simply  reduces to the operators with only temporal or spatial oscillations considered in Corollaries \ref{coro1} and \ref{coro2}, from which \eqref{thm1-re2} follows directly. Therefore, it suffices to consider the case $0<\va,\kappa<c_0$. Let $\rho$ be given by \eqref{ratio}. We divide the remaining proof into three cases: $\rho=0$; $0<\rho<\infty$; and $\rho=\infty$.

\noindent{\textbf{Case 1:} $0<\rho<\infty$.}
Assume that $(\rho/2)\leq \kappa/\va\leq 2\rho$ for $0<\va<c_0<1$. Then estimate \eqref{pthm1-1} implies that
\begin{align}\label{pthm1-2}
\Big(\fint_{Q_{\va}}|\na u_\va|^2\Big)^{1/2}\leq C \Big\{\Big(\fint_{Q_1}|\na u_\va|^2\Big)^{1/2} +  \Big(\fint_{Q_1}|F|^p\Big)^{1/p} \Big\},
\end{align} where $C$ depends only on $d$, $\mu$, $p$, $\rho$,  and $(\theta,L)$ in \eqref{cod3}.
 Let $w(x,t)=\va^{-1} u_\va(\va x, \va^2 t)$, it follows that
\begin{align*}
\pa_t w -\text{div} (A( \va x, \va^2t, x, t\va^2/\ka^2 ) \na w) =\widetilde{F} \quad\,\,\text{in } Q_{\va^{-1}},
\end{align*} where $\widetilde{F}(x,t)= \va F(\va x, \va^2t).$
Note that  $(\rho/2)\leq \kappa/\va\leq 2\rho$, \eqref{cod33} implies that $\mathcal{A} (x,t)=A( \va x, \va^2t, x, t\va^2/\ka^2 ) $ is uniformly H\"{o}lder continuous in $(x,t)$. Thanks to the standard Lipshchitz estimates of parabolic equations with H\"{o}lder continuous coefficients,
 \begin{align}\label{pthm1-2-1}
\begin{split}
|\na u_\va(0,0)|&=|\na w (0,0)| \leq C \Big\{\Big(\fint_{Q_1}|\na w|^2\Big)^{1/2} +\Big(\fint_{Q_1}| \widetilde{ F} |^p\Big)^{1/p} \Big\}\\
&\leq C \Big\{\Big(\fint_{Q_{\va }}|\na u_\va |^2\Big)^{1/2} +   \va^{1-\frac{d+2}{p}}\Big(\fint_{Q_1}|  F |^p\Big)^{1/p} \Big\}
\end{split}
 \end{align} for some positive constant $C$ depending only on $d$, $\mu$, $p$, $\rho$, and $(\theta,L)$ in \eqref{cod33},
from which and \eqref{pthm1-2},  the desired estimate \eqref{thm1-re2} follows directly.

\noindent{\textbf{Case 2:} $\rho=0$.}  In this case, $\va+\kappa\simeq \va$ for $\va<c_0$. Thus \eqref{pthm1-1} also reduces to \eqref{pthm1-2}.
 Let $w(x,t)=\va^{-1} u_\va(\va x, \va^2 t)$, it follows that
\begin{align*}
\pa_t w -\text{div} (A( \va x, \va^2t, x, t/(\va')^2 ) \na w) =\widetilde{F} \quad\,\,\text{in } Q_{\va^{-1}},
\end{align*} where $\widetilde{F}(x,t)= \va   F(\va x, \va^2 t)$ and $\va'=(\kappa/\va)^2.$
 Note that $B(x,t,s)=A( \va x, \va^2t, x, s)$ satisfies the condition \eqref{coro1con}. We therefore conclude from
 Corollary \ref{coro1} that
\begin{align}\label{pthm1-3}
\begin{split}
|\na u_\va(0,0)|&=|\na w (0,0)| \leq C \Big\{\Big(\fint_{Q_1}|\na w|^2\Big)^{1/2} +\Big(\fint_{Q_1}| \widetilde{ F} |^p\Big)^{1/p} \Big\}\\
&\leq C \Big\{\Big(\fint_{Q_{\va}}|\na u_\va |^2\Big)^{1/2} + \va^{1-\frac{d+2}{p}}\Big(\fint_{Q_1}|  F |^p\Big)^{1/p} \Big\},
\end{split}
 \end{align}
which, together with \eqref{pthm1-2}, gives \eqref{thm1-re2}.

\noindent{\textbf{Case 3:} $\rho=\infty$.}
 Now  estimate \eqref{pthm1-1} reduces to
\begin{align}\label{pthm1-5}
\Big(\fint_{Q_\ka}|\na u_\va|^2\Big)^{1/2}\leq C \Big\{\Big(\fint_{Q_1}|\na u_\va|^2\Big)^{1/2} +  \Big(\fint_{Q_1}|F|^p\Big)^{1/p} \Big\}.
\end{align}
 By setting $w(x,t)=\ka^{-1} u_\va(\ka x, \ka^2 t)$, we get
\begin{align}\label{pthm1-7}
\pa_t w -\text{div} (A(\ka x, \ka^2 t, x/(\va/\ka)^2, t) \na w) =\widetilde{F} \quad\,\,\text{ in }  Q_{\ka^{-1}}
\end{align} with $\widetilde{F}(x,t)=\ka F(\ka x,\ka^2 t).$ Since $B(x,t,y)=A(\ka x,\ka^2 t,y, t)$ satisfies  \eqref{coro2con}.
Corollary \ref{coro2}, together with \eqref{pthm1-5}, implies that
\begin{align}\label{pthm1-8}
\begin{split}
|\na u_\va(0,0)|&=|\na w(0,0)|\leq C \Big\{\Big(\fint_{Q_1}|\na w|^2\Big)^{1/2} +\Big(\fint_{Q_1}|\widetilde{F}|^p\Big)^{1/p} \Big\}\\
& \leq C \Big\{\Big(\fint_{Q_{\ka }}|\na u_\va |^2\Big)^{1/2} + \ka^{1- \frac{d+2}{p}}\Big(\fint_{Q_1}|  F |^p\Big)^{1/p} \Big\}\\
&\leq C \Big\{\Big(\fint_{Q_1}|\na u_\va|^2\Big)^{1/2} +  \Big(\fint_{Q_1}|F|^p\Big)^{1/p} \Big\},
\end{split}
\end{align}which is exactly \eqref{thm1-re2}.
By combining the estimates from Case 1 to Case 3, one derives \eqref{thm1-re2} immediately. The proof is complete.
\end{proof}

\begin{proof}[\textbf{Proof of Theorem 1.2}] Since $\pa_t+\mathfrak{L}_\va=\pa_t +\mathcal{L}_\va^\lm $ and $(1+\sqrt{\lm})\va=\va+\kappa$ when $\lm=\kappa^2/\va^2$. The large scale estimate \eqref{b-thm2-re1} follows directly from \eqref{b-thmapre2} by setting $\lm=\ka^2/\va^2.$
The proof of \eqref{b-thm2-re2}, based on Corollaries \ref{b-coro1} and \ref{b-coro2}, is almost the same as the proof of  \eqref{thm1-re2}. We therefore omit the details.
\end{proof}

\section{Proof of Theorem \ref{cothm}}

To begin with, we consider the convergence rate for the initial-Dirichlet problem
\begin{equation}\label{eq71}
\pa_tu^\lm_\va +\mathcal{L}^\lm_\varepsilon u^\lm_\va=F   \,\, \,\,\text{in } \Omega_T, \quad\quad
u^\lm_\va   =g    \,\, \,\,\text{on } \partial_p \Omega_T,
\end{equation} where $\Omega$ is a bounded Lipschitz domain in $\R^d$, $\Omega_T=\Omega\times (0,T),$ and  $\pa_p \Omega_T$ is the parabolic boundary of $\Omega_T$.
Let $\pa_t+\mathcal{L}_0^\lm$ be the homogenized operator of $\pa_t+\mathcal{L}^\lm_\varepsilon$, and $u_0^\lm$ the solution to
\begin{equation}\label{heq71}
\pa_tu^\lm_0+\mathcal{L}^\lm_0 u^\lm_0=F   \,\, \,\,\text{in } \Omega_T, \quad\quad
u^\lm_0  =g    \,\, \,\,\text{on } \partial_p \Omega_T.
\end{equation}
Similar to \eqref{def-w}, define
\begin{align} \label{def-wwan}
\begin{split}
\widetilde{w}_\varepsilon^\lm &=u_\varepsilon^\lm-u_0^\lm -\varepsilon S_\de ( (\chi^\lm)^\varepsilon\na  u_0^\lm)\eta_\de   +\varepsilon^{2}  S_\de ( (\pa_{x_i }\mathfrak{B}^\lm_{i (d+1) j})^\varepsilon\pa_{x_j} u_0^\lm)\eta_\de \\
&\quad+ \varepsilon^{2}  S_\de ( (\mathfrak{B}^\lm_{i (d+1) j})^\varepsilon \pa_{x_i } \pa_{x_j} u_0^\lm)\eta_\de  + \varepsilon^{2}  S_\de ( (\mathfrak{B}^\lm_{i (d+1) j})^\varepsilon\pa_{x_j} u_0^\lm) \pa_{x_i } \eta_\de ,
 \end{split}
\end{align} where $f^\va(x, t)=f(x,t,x/\va, t/\va^2), \de=(1+\sqrt{\lm})\va$,  $\chi^\lm$ and $ \mathfrak{B}^\lm$ are the matrices of correctors and flux correctors defined in Section 2.
The cut-off function $\eta_\de   \in C_c^\infty(\R^{d+1})$ satisfies
\begin{align}\label{etade}
 \begin{split}
 &0\leq \eta_\de  \leq 1,\quad \text{and}\quad \eta_\de  =1  \, \text{ in } \Omega_{T}\!\setminus\!\Omega_{T,4\de},\,\, \eta_\de  =0 \,\text{ in } \Omega_{T,2\de}, \\
& |\na \eta_\de |\leq C \de^{-1}, \quad |\pa_t \eta_\de |+ |\nabla^2 \eta_\de  |\leq C \de^{-2},
\end{split}\end{align}
where $\Omega_{T, \de }$ denotes the (parabolic) boundary layer
\begin{equation}\label{layer}
\Omega_{T, \de} =
\left( \big\{ x\in \Omega: \, \text{\rm dist} (x, \partial\Omega)<\de\big\}
\times (0, T)\right)
\cup \big( \Omega \times (0, \de^2)\big)
\end{equation}
for $0< \de< c$.

\begin{theorem}\label{cothm71}
Assume that $A$ satisfies \eqref{cod1}, \eqref{cod2} and \eqref{cod3} with $\theta=1$. Let $\widetilde{w}_\va^\lm$ be defined as \eqref{def-wwan}. Then, for any $\psi\in L^2(0, T; H_0^1(\Omega))$,
\begin{equation}\label{cothm71re1}
\aligned
&\Big| \int_0^T \langle (\pa_t +\mathcal{L}^\lm_\va)  \widetilde{w}^\lm_\va, \psi\rangle_{H^{-1}(\Omega) \times H^1_0(\Omega)} dt   \Big| \\
& \le  C \big\{ \| u_0^\lambda \|_{L^2(0, T; H^2(\Omega))}
+ \|\partial_t u_0^\lambda\|_{L^2(\Omega_T)} \big\}
\big\{ \delta  \|\nabla \psi \|_{L^2(\Omega_T)}
+ \delta^{1/2}
\| \nabla \psi\|_{L^2(\Omega_{T, 5\delta})}
\big\},
\endaligned
\end{equation} where $C$ depends only on $d$, $\mu$, $L$, $\Omega$ and $T$.
In particular,
\begin{align}\label{cothm71re2}
\|\na \widetilde{w}_\va^\lm\|_{L^2(\Omega_T)} &\leq  C \de^{1/2}\big (\|\na^2 u_0^\lm\|_{L^2(\Omega_T)} +\|\pa_t u_0^\lm\|_{L^2(\Omega_T)}+\|\na  u\|_{L^2(\Omega_T)}\big).
\end{align}
\end{theorem}

\begin{proof}
Performing similar analysis as in Lemma \ref{leap-2},  parallel to \eqref{pl316}, it is not difficult to prove that
\begin{align}\label{plcothm71-1}
\begin{split}
& \Big| \int_0^T\big\langle (\pa_t +\mathcal{L}^\lm_\va)  \widetilde{w}^\lm_\va, \psi\rangle_{H^{-1}(\Omega) \times H^1_0(\Omega)}  dt   \Big|\\
  &\leq C  \|\na u^\lm_0\|_{L^2(\Omega_{T,5\de})} \|\na \psi\|_{L^2(\Omega_{T,5\de})}+ C\de   \|\na u^\lm_0\|_{L^2(\Omega_T)} \|\na \psi\|_{L^2(\Omega_T)}\\
 &\quad+C\de   \big\{ \|\na^2 u^\lm_0\|_{L^2( \Omega_T)} +\|\pa_t u^\lm_0\|_{L^2( \Omega_T)}  \big\} \|\na \psi\|_{L^2(\Omega_T)}.
 \end{split}
\end{align}
Thanks to Lemma 7.1 in \cite{gsamar2020}, we know that
\begin{align}\label{plcothm71-2}
\|\na u\|_{L^2(\Omega_{T,\de})}\leq C \de^{1/2} \big\{ \|\na^2 u \|_{L^2( \Omega_T)} +\|\pa_t u\|_{L^2( \Omega_T)} +\|\na  u\|_{L^2(\Omega_T)} \big\},
\end{align}
which, together with \eqref{plcothm71-1}, gives \eqref{cothm71re1}. Note that $\int_0^T\big\langle \pa_t   \widetilde{w}^\lm_\va, \widetilde{w}^\lm_\va \rangle_{H^{-1}(\Omega) \times H^1_0(\Omega)}  dt\geq 0.$ By taking $\psi=\widetilde{w}_\va^\lm$ in \eqref{plcothm71-1}, and using \eqref{plcothm71-2}, one gets \eqref{cothm71re2} immediately.
The proof is complete.
\end{proof}

\begin{theorem}\label{cothm72}
Assume that $A$ satisfies  \eqref{cod1},  \eqref{cod2} and  \eqref{cod3} with $\theta=1$.
Let $\Omega$ be a bounded $C^{1, 1}$ domain in $\mathbb{R}^d$.
Let $u^\lm_\va, u^\lm_0$ be respectively the solutions to \eqref{eq71} and \eqref{heq71}.
Then
\begin{equation}\label{cothm72re}
\| u^\lm_\va  -u^\lm_0\|_{L^2(\Omega_T)}
\le C \delta
\big\{
\| u^\lm_0 \|_{L^2(0, T; H^2(\Omega))}
+ \|\partial_t u^\lm_0\|_{L^2(\Omega_T)}
 \big\},
\end{equation}
where $C$ depends only on $d$, $\mu$, $L$, $\Omega$ and $T$.
\end{theorem}
\begin{proof}
The proof is based on \eqref{cothm71re1}, \eqref{cothm71re2} and the standard duality argument initiated in \cite{suslinaD2013}, see also \cite{gsjfa2017} for the parabolic settings.
For $G\in L^2(\Omega_T)$, let $v_\va^\lm(x,t)$ and $v^\lm_0(x,t)$ be respectively the solutions to
\begin{align*}
-\pa_t v_\va^\lm +\mathcal{L}^{\lm*}_\va v_\va^\lm=G \,\,\text{ in  } \Omega_T, \quad\quad v^\lm_\va=0  \,\, \text{ on } \Gamma_T\cup \big(\Omega\!\times\! \{t\!=\!T\}\big) ,
\end{align*} and
\begin{align*}
-\pa_t v_0^\lm +\mathcal{L}^{\lm*}_0 v^\lm_0=G \,\,\text{ in  } \Omega_T, \quad\quad v^\lm_0=0  \,\, \text{ on }\Gamma_T\cup \big(\Omega\!\times\! \{t\!=\!T\}\big),
\end{align*}
where  $ \Gamma_T= \pa \Omega \times(0,T) $, and $\mathcal{L}^{\lm *}_\va=- \text{div} ((\widehat{A^{\lm*}})(x,t)\na )$ with $ A^{\lm *}$ being the adjoint of $A^\lm$.
Then $v_\va^\lm(x,T-t), v_0^\lm(x,T-t)$ satisfy \eqref{eq71} and \eqref{heq71} with $g\equiv0$, and $A^{\lm*}, \widehat{A^{\lm*}}$ replaced by $A^{\lm*}(x,T-t,x/\va, (T-t)/\va^2)$ and $\widehat{A^{\lm*}}(x,T-t) $ respectively.
Let  $ \widetilde{\chi}^\lm, \widetilde{\mathfrak{B}}^\lm$ be respectively the correctors and flux correctors of the operators
$\pa_t-\text{div} (A^{\lm*}(x,T-t,x/\va,(T-t)/\va^2) \na ).$
Define
\begin{align}\label{p-cothm72-1}
\begin{split}
z^\lm_\va(x,t)&= v_\va^\lm(x,T-t)- v_0^\lm(x,T-t)-\varepsilon S_\de ( (\widetilde{\chi}^\lm)^\varepsilon\na  v_0^\lm(x,T-t))\widetilde{\eta}_\de\\
&\quad   +\varepsilon^{2}  S_\de ( (\pa_{x_i }\widetilde{\mathfrak{B}}^\lm_{i (d+1) j})^\varepsilon\pa_{x_j} v_0^\lm(x,T-t))\widetilde{\eta}_\de\\
&\quad  + \varepsilon^{2}  S_\de ( (\widetilde{\mathfrak{B}}^\lm_{i (d+1) j})^\varepsilon \pa_{x_i } \pa_{x_j} v_0^\lm(x,T-t))\widetilde{\eta}_\de\\
&\quad  + \varepsilon^{2}  S_\de ( (\widetilde{\mathfrak{B}}^\lm_{i (d+1) j})^\varepsilon\pa_{x_j} v_0^\lm(x,T-t)) \pa_{x_i } \widetilde{\eta}_\de ,
\end{split}
\end{align}where $\de=(1+\sqrt{\lm}) \va$, the cut-off function $\widetilde{\eta}_\de \in C^\infty_c (\R^{d+1})$ satisfies
\begin{align*}
 \begin{split}
 &0\leq \widetilde{\eta}_\de  \leq 1,\quad \text{and}\quad \widetilde{\eta}_\de  =1  \, \text{ in }  \Omega _{T}\!\setminus\!\Omega_{T,10\de},\,\, \widetilde{\eta}_\de  =0 \,\text{ in }  \Omega _{T,8\de}, \\
& |\na \widetilde{\eta}_\de |\leq C \de^{-1}, \quad |\pa_t \widetilde{\eta}_\de |+ |\nabla^2 \widetilde{\eta}_\de  |\leq C \de^{-2}.
\end{split}\end{align*}
In view of \eqref{cothm71re2}, we have
\begin{align}\label{p-cothm72-2}
\begin{split}
 \|\na z_\va^\lm\|_{L^2(\Omega_T)} &\leq  C \de^{1/2}\big (\|\na^2 v_0^\lm\|_{L^2(\Omega_T)} +\|\pa_t v_0^\lm\|_{L^2(\Omega_T)}+\|\na v_0^\lm\|_{L^2(\Omega_T)}\big)\\
 &\leq C \de^{1/2} \|G\|_{L^2(\Omega_T)},
 \end{split}
\end{align}
where we have used the $H^2$ estimate for parabolic systems with Lipschitz coefficients \cite{dongtran2018}.

We now perform the duality argument. Let $\widetilde{w}^\lm_\varepsilon$ be given by \eqref{def-wwan}.
Observe that
\begin{align}\label{p-cothm72-3}
 \big|\int_{\Omega_T}\widetilde{w}^\lm_\varepsilon\,   G\, dxdt \big|&=\big| \int_0^T\langle (\pa_t+\mathcal{L}_\va^\lm) \widetilde{w}^\lm_\varepsilon , ~ v^\lm_\varepsilon(t)   \rangle dt \big| \nonumber\\
&\leq \big|\int_0^T\langle(\pa_t+\mathcal{L}_\va^\lm) \widetilde{w}^\lm_\varepsilon , ¡¤z^\lm_\va(T-t) \rangle dt \big| +\big|\int_0^T\langle (\pa_t+\mathcal{L}_{\varepsilon}^\lm ) \widetilde{w}^\lm_\varepsilon , ~v^\lm_0 (t)   \rangle dt\big|\nonumber\\
&\quad+ \big|\int_0^T\langle (\pa_t+\mathcal{L}_{\varepsilon}^\lm ) \widetilde{w}^\lm_\varepsilon , ~v^\lm_\varepsilon (t)-v^\lm_0 (t)  - z^\lm_\va(T-t) \rangle dt\big| \nonumber\\
&\doteq J_1+J_2+J_3.\end{align}
Thanks to \eqref{cothm71re1} and \eqref{p-cothm72-2}, we have
\begin{align*}
J_1\leq C \de \big\{
\| u^\lm_0 \|_{L^2(0, T; H^2(\Omega))}
+ \|\partial_t u^\lm_0\|_{L^2(\Omega_T)}
 \big\} \|G\|_{L^2(\Omega_T)}.
\end{align*}
By \eqref{cothm71re1}, and \eqref{plcothm71-2} for $v_0^\lm$, we obtain that
\begin{align*}
J_2\leq C \de  \big\{
\| u^\lm_0 \|_{L^2(0, T; H^2(\Omega))}
+ \|\partial_t u^\lm_0\|_{L^2(\Omega_T)}
 \big\} \|G\|_{L^2(\Omega_T)}.
\end{align*}
Finally, in view of \eqref{p-cothm72-1} and Lemma \ref{le2.0}, we know that $v^\lm_\varepsilon (t)-v^\lm_0 (t)- z^\lm_\va(T-t)$ is supported in $\Omega_T\!\setminus\!\Omega_{T,6\de}$ and
\begin{align*}
 \| \na (v^\lm_\varepsilon (t)-v^\lm_0 (t)- z^\lm_\va(T-t))\|_{L^2(\Omega_T)}\leq C \|\na v^\lm_0\|_{L^2(\Omega_T)}\leq C \|G\|_{L^2(\Omega_T)}.
\end{align*}This, together with \eqref{cothm71re1}, implies that
\begin{align}
J_3\leq  C \de  \big\{
\| u^\lm_0 \|_{L^2(0, T; H^2(\Omega))}
+ \|\partial_t u^\lm_0\|_{L^2(\Omega_T)}
 \big\} \|G\|_{L^2(\Omega_T)}.
\end{align}
By taking the estimates for $J_1$-$J_3$ into \eqref{p-cothm72-3}, we derive  that
\begin{align*}
\|\widetilde{w}^\lm_\va\|_{L^2(\Omega_T)}\leq C \de  \big\{
\| u^\lm_0 \|_{L^2(0, T; H^2(\Omega))}
+ \|\partial_t u^\lm_0\|_{L^2(\Omega_T)}
 \big\} \|G\|_{L^2(\Omega_T)},
\end{align*}
which, combined with the estimate
\begin{align*}
 \| u^\lm_\varepsilon -u^\lm_0  - \widetilde{w}^\lm_\va\|_{L^2(\Omega_T)}\leq C\de \|\na u^\lm_0\|_{L^2(\Omega_T)},
\end{align*}
gives \eqref{cothm72re}.
\end{proof}
Now we are ready to prove Theorem \ref{cothm}.
\begin{proof}[\textbf{Proof of Theorem \ref{cothm}.}]
 Let $u_\va $ be the weak solution of $\pa_t u_\va +\mathfrak{L}_\va u_\va=F$ in $\Omega_T$ with $u_\va=g $ on $\pa_p \Omega_T$, and
$u_0$ the solution of the homogenized problem $\pa_t u_0-\text{\rm div}(\widehat{A}(x,t) \nabla u_0)=F$ in $\Omega_T$ with $u_0=g$ on $\pa_p \Omega_T$.
Let $\lm=\kappa^2/\va^2$. Then $\pa_tu_\va+\mathcal{L}_\va^\lm u_\va=\pa_tu_\va+ \mathfrak{L}_\va  u_\va=F$ in $\Omega_T$.
Let $u^\lm_0$ be the solution of  $\pa_t u_0^\lm-\text{\rm div}(\widehat{A^\lambda}(x,t) \nabla u^\lm_0)=F$ in $\Omega_T$
with $u^\lm_0=g$ on $\pa_p\Omega_T$.
Note that
\begin{equation}\label{3-10-0}
\aligned
\| u_\va -u_0\|_{L^2(\Omega_T)}
 & \le \| u_\va - u_0^\lm\|_{L^2(\Omega_T)}
+\| u^\lm_0  -u_0\|_{L^2(\Omega_T)}\\
& \le
C  (\kappa +\va) \big\{ \| u_0\|_{L^2(0, T; H^2(\Omega))}
+\|\partial_t u_0\|_{L^2(\Omega_T)} \big\}
+ \| u^\lm_0  -u_0\|_{L^2(\Omega_T)},
\endaligned
\end{equation}
where we have used Theorem \ref{cothm72} for the last inequality.
To estimate $u^\lm_0  -u_0$, we observe that $u^\lm_0  -u_0=0$ on $\pa_p\Omega_T$ and
$$
\pa_t(u^\lm_0 -u_0)-\text{\rm div} (\widehat{A^\lm} \nabla (u_0^\lm -u_0))
= \text{\rm div} ( (\widehat{A^\lm} -\widehat{A})\nabla u_0 ) \quad \text{ in }\Omega_T.
$$
By energy estimates,
$$
\aligned
\|\na u_0-\na u^\lm_0 \|_{L^2(\Omega_T)}
 & \le C  \|\widehat{A^\lm}-\widehat{A}\|_{\infty} \| \nabla u_0 \|_{L^2(\Omega_T)},
\endaligned
$$
where $C$ depends only on $d$, $\mu$, $\Omega$, and $T$.
This, together with Lemma \ref{lemconA}, (\ref{3-10-0}) and Poincar\'{e}'s inequality, gives (\ref{cothmre}).
\end{proof}
\noindent\textbf{Acknowledgments.} The authors would like to thank Professor Zhongwei Shen for bringing this problem to them, and also for the guidance as well as the enlightening discussions.


\bibliographystyle{amsplain}

\bibliography{gn2020}

\vspace{1cm}
\noindent Jun Geng \\
School of Mathematics and Statistics,  Lanzhou University, Lanzhou, P. R. China\\
E-mail: gengjun@lzu.edu.cn
\vspace{1cm}\\
\noindent Weisheng Niu \\
School of Mathematical Science, Anhui University,
Hefei, 230601, P. R. China\\
E-mail: niuwsh@ahu.edu.cn\\

\noindent\today

 \end{document}